\documentclass[12pt]{amsart}
\usepackage{amssymb}
\usepackage{amsmath}
\usepackage[latin1]{inputenc}
\usepackage{color}
\DeclareMathOperator*{\minn}{min}

\def\Z{{\mathbb Z}}\def\T{{\mathbb T}}\def\R{{\mathbb R}}\def\C{{\mathbb C}}

\def\CC{{\mathcal C}}\def\LL{{\mathcal L}}
\def\GG{{\mathcal G}}
\def\EE{{\mathcal E}}

\def\OO{{\mathcal O}}
\def\MM{{\mathcal M}}\def\PP{{\mathcal P}}

\def\bC{{\mathbb C}}

\def\cC{{\mathcal C}}\def\cL{{\mathcal L}}
\def\cB{{\mathcal B}}\def\cG{{\mathcal G}}
\def\cT{{\mathcal T}}
\def\cD{{\mathcal D}}
\def\cH{{\mathcal H}}
\def\cM{{\mathcal M}}

\def\cU{{\mathcal U}}

\def\f{\varphi}\def\o{\omega}
\def\g{\gamma}\def\ga{\gamma}
\def\s{\sigma}\def\k{\kappa}
\def\al{\alpha}\def\be{\beta}\def\ep{\varepsilon}
\def\be{\beta}\def\t{\tau}\def\la{\lambda}\def\m{\mu}

\def\L{\Lambda}\def\G{\Gamma}

\def\i{\infty}
\def\p{\partial}

\def\Leb{\mathrm{Leb}}

\def\b#1{\lbrace#1\rbrace}

\def\aa#1{\left\Vert#1\right\Vert}

\def\l#1{\langle #1\rangle}

\def\lan{\langle}
\def\ran{\rangle}
\def\<{\langle}
\def\>{\rangle}

\def\sm{\setminus}
\def\lsim{\lesssim}

\def\Lin{\mathrm{Lin}}

\def\w{\mathrm{anal}}

\theoremstyle{plain}
\newtheorem{Thm}{Theorem}[section]

\newtheorem{Prop}[Thm]{Proposition}
\newtheorem{Lem}[Thm]{Lemma}
\newtheorem{Cor}[Thm]{Corollary}

\newtheorem*{Rem}{Remark}

\newtheorem*{Cor*}{Corollary}
\newtheorem*{Prop*}{Proposition}

\newtheorem{Main}{Theorem}

\newtheorem*{Principal}{Theorem}

\newtheorem*{Russ}{Theorem (R\"ussmann)}

\def\cC{\mathcal C}

\def\carre{ \hfill $\Box$    }

\def\a{\alpha}
\def\w{\omega}

\def\pa{\partial}
\def\s{\sigma}

\let\newpf\proof \let\proof\relax

\newcommand{\ba}{\overline{A}}

\newcommand{\cO}{\mathcal{O}}

\def\be{\begin{equation}}
\def\ee{\end{equation}}

\def\ba{{\begin{align}}}
\def\ea{{\end{align}}}

\def\bm{\begin{matrix}}
\def\em{\end{matrix}}

\def\a{{\alpha}}

\def\cL{{\mathcal L}}

\def\g{{\gamma}}

\def\bD{\mathbb{D}}
\def\bN{\mathbb{N}}
\def\0{{\mathbf 0}}

\def\cal{\mathcal}

\newtheorem{thm}{Theorem}[section]

\newtheorem{lemma}[thm]{Lemma}

\theoremstyle{remark}

\def\cO{\mathcal{O}}

\def\cB{\mathcal{B}}
\def\cC{\mathcal{C}}

\theoremstyle{definition}

\newenvironment{proof}{ \noindent{\it Proof.}\quad}{\ \hfill $\Box$\vskip .2cm}

\def\sm{\smallsetminus}

\def\ssm{\smallsetminus}

\renewcommand{\setminus}{\ssm}

\newcommand{\id}{\operatorname{id}}

\newcommand{\eps}{{\epsilon}}

\newcommand{\de}{{\delta}}

\newcommand{\Om}{{\Omega}}
\newcommand{\om}{{\omega}}

\newcommand{\N}{{\mathbb N}}

\newcommand{\bR}{{\mathbb R}}
\newcommand{\bT}{{\mathbb T}}
\newcommand{\bZ}{{\mathbb Z}}

\def\B0{{\bold{0}}}

\def\be{\begin{equation}}
\def\ee{\end{equation}}

\catcode`\@=12

\def\Empty{}
\newcommand\oplabel[1]{
  \def\OpArg{#1} \ifx \OpArg\Empty {} \else
  	\label{#1}
  \fi}
		
\newcommand{\comm}[1]{}
\newcommand{\comment}[1]{}

\begin{document}

\title{Around the stability of KAM-tori}
\author{L. H. Eliasson, B. Fayad, R. Krikorian}
\address{
IMJ-PRG University Paris-Diderot\\
IMJ-PRG CNRS\\
LPMA UPMC}
\email{hakan.eliasson@math.jussieu.fr, bassam@math.jussieu.fr, raphael.krikorian@upmc.fr }

\date{\today}

\thanks{Supported by  ANR-10-BLAN 0102}

\maketitle

\begin{abstract} 

$ \ $ 
We study the accumulation of an invariant quasi-periodic torus  of a hamiltonian flow by other quasi-periodic invariant tori. 


We show that an analytic invariant torus $\cT_0$ with Diophantine frequency $\o_0$ is never isolated due to the following alternative.   If the  Birkhoff normal form of the Hamiltonian  at $\cT_0$ satisfies a R\"ussmann transversality condition, the  torus $\cT_0$ is accumulated by KAM tori of positive total measure. If the Birkhoff normal form is degenerate,  there exists a subvariety of dimension at least $d+1$ that is foliated by analytic invariant tori with frequency $\o_0$.


For frequency vectors $\o_0$ having a finite uniform Diophantine exponent  (this includes a residual set of Liouville vectors), we show that if the Hamiltonian $H$ satisfies  a Kolmogorov non degeneracy condition at $\cT_0$, then $\cT_0$ is accumulated by KAM tori of positive total measure.


In $4$ degrees of freedom or more, we construct for any $\o_0 \in \R^d$,  $C^\infty$ (Gevrey) Hamiltonians $H$ with a smooth invariant torus $\cT_0$  with frequency $\o_0$ that is not  accumulated by a positive measure of invariant  tori. 

\end{abstract}

\tableofcontents 

\section{Introduction}\label{sIntro}

Let 
\begin{equation} H(\f,r)= \l{\o_0,r}+\cO(r^2) \label{HH} \end{equation} 
be a $C^2$ function
defined for $\f\in \bT^d=\bR^d/\bZ^d$ and $r\sim 0\in\bR^d$. 

The {\it Hamiltonian system} associated to $H$ is given by 
$$(*)_H\quad\left\{\begin{array}{l}
\dot \f=\pa_{r}H(\f,r)\\ 
\dot r=-\pa_{\f}H(\f,r).\end{array}\right.$$
Clearly the torus $\bT^d\times\b{0}$ is invariant under the Hamiltonian flow
and the induced dynamics is the translation
$$(t,\f)\mapsto \f+t\o_0.$$
Moreover this torus is Lagrangian with respect to the 
canonical symplectic form $d\f\wedge dr$ on $\bT^d\times \bR^d$.

The objective of this paper is to investigate the "KAM stability" of the torus $\cT_0:=\bT^d\times\b{0}$ under different hypothesis on $H$ and $\o_0$. 
We first explain what we understand by "KAM stability". In its stronger form, we use this  terminology to refer to the classical KAM (after Kolmogorov, Arnol'd, Moser) phenomenon of accumulation of 
$\cT_0$ by invariant tori whose Lebesgue density in the phase space tend to one in the neighborhood of $\cT_0$ and whose frequencies cover a set of positive measure. More precisely, a vector $\o$ is said to be {\it Diophantine} if there exist
$\k>0,\ \t>d-1$
 such that
\begin{equation} |\<k,\o\>|\ge \frac{\k}{|k|^{\t}}\quad \forall k\in\Z^d\sm \{0\} \label{cd} \end{equation}
We then use the notation $\omega \in DC(\k,\t)$ where $\tau$ is the Diophantine {\it exponent}    of $\omega$ and $\kappa$ its Diophantine constant.   We say that a $C^r$ (or smooth, or  analytic) invariant Lagrangian torus with an induced flow
that is $C^r$ (or smoothly, or analytically)  conjugated to a Diophantine translation 
$$(t,\f)\mapsto \f+t\o$$
is a $C^r$ (smooth, analytic)  {\it KAM-torus} of $(*)_H$ with  
{\it translation vector} $\o$.

We still say that $\cT_0$ is "KAM stable", in a slightly weaker sense, if we drop the requirement that the frequencies cover a set of positive measure. 

When we prove the accumulation of $\cT_0$ by invariant KAM tori but we do not know if their measure is positive we simply say that $\cT_0$ is accumulated by KAM tori and do not speak of stability.

In this paper we deal essentially with the following situations and results. Unless otherwise mentioned the Hamiltonian $H$ is assumed to be analytic as well as $\cT_0$ and the KAM tori that are obtained.  The exact statements and notations will be deferred to the next section.

\begin{itemize}
\item[(i)] If $\o_0$ is Diophantine then $\cT_0$ is accumulated by KAM tori.  
\item[(ii)]  If $\o_0$ is Diophantine and if the  Birkhoff normal form (BNF) of $H$ satisfies a R\"ussmann transversality condition at $\cT_0$ (see Section \ref{analy}), then $\cT_0$ is KAM stable. 
\item[(iii)] In two degrees of freedom ($d=2$), if $\o_0$ has  rationally independent coordinates and if $H$ satisfies  a Kolmogorov non degeneracy condition of its Hessian matrix at $\cT_0$ (see Section \ref{liouv}), $\cT_0$ is KAM stable.
    For $d\geq 3$, we get KAM stability for a class of $\o_0$ that includes all vectors except a meagre set of zero Hausdorff dimension. 
\item[(iv)] For $d\geq 4$, for any $\o_0 \in \R^d$, there exists a $C^\infty$ (Gevrey) $H$ as in \eqref{HH} such that $\cT_0$ is not KAM stable (no positive measure of accumulating tori). 
\item[(v)] For $d=2$, if $\o_0$ is Diophantine and $H$ is smooth $\cT_0$ is KAM stable. 
\end{itemize}

It was conjectured by M. Herman in his ICM98-lecture \cite{H}
that in the neighborhood of an analytic KAM-torus, the set of KAM-tori is of positive measure, i.e. KAM stability in a weak sense holds. (i) falls short of proving Herman's conjecture. In the case where we cannot prove that $\cT_0$ is KAM stable we actually show that there exists a subvariety of dimension at least $d+1$ that is foliated by analytic KAM tori with frequency $\o_0$.
The proof of (i) is based on a counter term KAM theorem inspired by Herman. For every value $c\sim 0$ of the action variable there exists a unique frequency $\Omega(c)$ that   cancels the counter term, and if this frequency is Diophantine this yields an invariant KAM torus with frequency $\Omega(c)$. One can show that the jets of the function $\Omega(c)$ are given by those of the gradient of the Birkhoff normal form when the latter is well defined (which is the case if $\o_0$ is Diophantine).
The following alternative then holds : either the BNF is non degenerate and the function $\Omega$ takes Diophantine values on a positive measure set which yields KAM stability (this is (ii)), or the BNF is degenerate and  we can use the analytic dependance of the counter term on the action variable to show the existence of a direction (after a coordinate change in the action variable) that spans a subvariety of invariant KAM tori of frequency $\o_0$.

Point (ii) is a more classical KAM result. Note however that one does not get KAM stability in 
the strong sense since the set of  frequencies under the  R\"ussmann transversality condition does not necessarily have positive measure. The proof of (ii) is obtained from the counter term KAM theorem as explained above and it can be adapted to smooth Hamiltonians. The hypothesis $\o_0$ Diophantine is necessary to guarantee the existence of a BNF.

In (iii) (strong) KAM stability is obtained in the neighborhood of a class of tori that are not necessarily Diophantine. The difficulty is that the BNF may not be defined. This difficulty can be overcome if the Kolmogorov non degeneracy condition is satisfied by $H$ at $\cT_0$, and if the rationally independent frequency $\o_0$ satisfies an arithmetic condition that contains all rationally independent vectors if $d=2$ and all but a meagre set of Hausdorff dimension 0 if $d\geq 3$. The condition is that the uniform Diophantine exponent of $\o_0$ denoted by $\widehat{\o} (\o_0)$ be finite. We recall that in the case of flows, we define $\widehat{\o} (\o_0)$ as the supremum of all real numbers $\gamma$ such that for any sufficiently  large $N$, there exists $k \in \Z^d-\{0\}$ such that $\|k\|\leq N$ and $|(k,\o)|\leq N^{-\gamma}$. We do not know whether invariant tori with frequencies $\o_0$ such that $\widehat{\o} (\o_0)=+\infty$ are KAM stable if the Kolmogorov non degeneracy condition is satisfied. 

The construction of (iv) is based on the successive conjugation method (Anosov-Katok construction \cite{AK}) starting from an "infinitely degenerate twist map" of the form $(\varphi,r) \mapsto (\varphi+f(r),r)$ with the frequency map $f$ such that $f(0)=\o_0$ and $f(r)$ having a fixed Liouville coordinate in small neighborhoods of any  $r$ such that $r_d\neq 0$. The construction only applies to the case $d\geq 4$. In case $d=2$ a smooth version of our  KAM counter term theorem proves the accumulation of a Diophantine torus $\cT_0$ by a positive measure set of Diophantine tori just as in Herman's last geometric theorem any Diophantine KAM circle of a smooth diffeomorphism of the annulus is shown to be accumulated by a positive measure set of KAM circles \cite{FK}.

\section{Statements}\label{sStatements}

\subsection{Analytic KAM tori are never isolated.} Let $H$ be a real analytic  function of the form \eqref{HH}.  \label{analy}

\begin{Main}\label{mPrincipal} If $\omega_0$ is Diophantine, the torus $\bT^d\times\b{0}$ is accumulated by 
analytic KAM tori of $(*)_H$ with Diophantine translation vector.
\end{Main}

In fact, we shall prove a more precise result. Let
$N_H$ be the {\it Birkhoff Normal Form} of $H$, that is a uniquely  defined formal power series in the $r$ variable as soon as $\omega_0$ is Diophantine (see Section \ref{s22}). 
We say that $N_H$ is $j$-{\it degenerate} if there exist
$j$ orthonormal vectors $\g_1,\dots,\g_j$ such that for every   $r\sim 0\in\bR^d$
$$\l{\p_rN_H(r),\g_i}=0\quad \forall\ 1\le i\le j,$$
but no $j+1$ orthonormal vectors with this property. Since
$\omega_0\not=0$ clearly $j\le d-1$. 
{ A $0$-degenerate $N_H$ is 
also said to be {\it non-degenerate}.}

\begin{Main}\label{mA}
If  $\o_0$ is Diophantine and $N_H$ is $j$-degenerate,
then there exists an analytic (co-isotropic) subvariety of dimension $d+j$ 
containing $\T^d\times\b{0}$ and foliated by analytic KAM-tori of $(*)_H$
with translation vector $\o_0$.
\end{Main}

A stronger result is known when $N_H$ is $(d-1)$-degenerate. Indeed 
R\"ussmann \cite{R} (in a different setting) proved

\begin{Russ}\label{Russmann}
{If  $\o_0$ is Diophantine and} $N_H$ is $(d-1)$-degenerate,
then a full neighborhood of
$\T^d\times\b{0}$ is foliated by analytic KAM-tori of $(*)_H$
with translation vector $\in\R\o_0$.
\end{Russ}

Our proof of Theorem \ref{mA} in Section \ref{sCo-isotropic}  will also yield R\"ussmann's result.  Theorem \ref{mPrincipal} follows from Theorem \ref{mA} in the degenerate case and from a more classical KAM theorem in the non-degenerate theorem that we discuss in the next section.

\subsection{KAM stability under non degeneracy conditions of the BNF}
\ 

{Let $H$ be a real analytic  function of the form \eqref{HH}. 
We say that $H$ has a normal form $N_H$ if there exists a formal power series $N_H$ 
and a formal symplectic mapping  $Z$ of the form
$$Z(\f,r)=(\f+\cO(r), r+\cO^2(r))$$
such that 
$$H \circ Z(\f,r)= N_H^q(r)+\OO^{q+1}(r)\in\CC^{\w}(\T^d\times \b{0}).$$

\begin{Rem} 
This is in particular the case when $\omega_0$ is Diophantine  --  $N_H$ is
the classical  
Birkhoff normal form. Moreover if a normal form  exists and $\om$ is rationally independent, then it is unique. \end{Rem}

Only assuming existence and non-degeneracy of the normal form $N_H$, 
we shall prove the following.

\begin{Main}\label{mB}
If $N_H$ exists, is unique  and  is non-degenerate, then  in any neighborhood of
$\T^d\times\b{0}$ the set of analytic KAM-tori of $(*)_H$ is of
positive Lebesgue measure with density one at the torus $\T^d \times \{0\}$. In particular, if $\o_0$ is Diophantine and if $N_H$ is non degenerate at $\cT_0$, then $\cT_0$ is KAM stable. 

\end{Main}

The condition that $N_H$ is non-degenerate is essentially equivalent to R\"ussmann's non-degeneracy
condition (see \cite{Russman2,You}). It is here shown to be sufficient in this singular perturbation
situation. We recall that the result does not state strong KAM stability since the 
 the frequencies of the KAM tori do not necessarily cover a set of positive measure.  Point (ii) of our introduction corresponds to the second statement of Theorem \ref{mB}.

Hence, the conjecture of M. Herman has an affirmative answer
when $N_H$ is { non-degenerate} (theorem \ref{mB}) or $(d-1)$-degenerate 
(R\"ussmann's theorem). Our theorems do not provide an answer 
to the conjecture in the intermediate cases.

\medskip

{ \subsection{KAM stability in the absence of BNF : Liouville torus with 
{ non-degeneracy of Kolmogorov type}} \label{liouv}

{Let $H$ be a real analytic  function of the form \eqref{HH}. 
and let 
$$M_0=\int_{\T^d} \frac{\partial^2 H }{\partial r^2 } (\varphi,0) d \varphi.$$ 
We recall the notation $\widehat{\o} (\o_0)$ as the supremum of all real numbers $\gamma$ such that for any sufficiently  large $N$
$$\min_{0<|k|\leq N}|\l{k,\o_0}|\leq N^{-\gamma}.$$
}
\begin{Main} \label{liouville.twist} If $\widehat{\o} (\o_0)<+\infty$  and if 
{
$M_0$ is non-singular}  then in any neighborhood of
$\T^d\times\b{0}$ the set of analytic KAM-tori of $(*)_H$ is of
positive Lebesgue measure with density one at the torus $\T^d \times \{0\}$. Moreover, 
the set of frequencies of the KAM tori has positive Lebesgue measure in $\R^d$.
\end{Main}

Since for any rationally independent vector $\o_0 \in \R^2$ we have that  $\widehat{\o} (\o_0)=1$ we see that KAM stability holds at $\cT_0$ without any other arithmetic condition { when $d=2$}. 
{This is a precise formulation of (iii).}

\begin{Rem} We could relax the condition $\widehat{\o} (\o_0)<+\infty$ to  the  existence of sequences $Q_n\to \infty$ and $\eps_n \to 0$  such that $|(k,\omega_0)|\geq e^{-\eps_n Q_n}$ for any $k \in \Z^d, 0<|k|\leq Q_n$. 
\end{Rem}

\medskip

\subsection{Smooth counterexamples to KAM stability.}

In the $\cC^{\i}$-cate\-gory the situation is different from that of Theorem \ref{mPrincipal}. 
For $d=2$, we show in section \ref{secd2} that the same 1-dimensional phenomenon of the frequency map pointed out by Herman (see  \cite{FK} for the discrete case) gives a set  of positive    
measure of $\cC^\i$ KAM-tori in any neighborhood of $\T^2\times\b{0}$.  
For $d=3$, we have no results, but for $d\ge4$ we shall prove

\begin{Main}\label{mC}
Let $d\ge 4$. For any $\eps>0, s \in \N$ there exists a function $h$ in $C^\infty(\T^4\times\R^4)$, satisfying  
$h(\f,r)=\cO^\infty (r_4)$ and
{   
$${\|h\|}_{\cC^s(\T^4\times\R^4)}<\eps,$$ 
}
such that the flow $\Phi_H^t$ of $H(\f,r)=(\omega_0,r)+h(\f,r)$ satisfies 
$$ \limsup_{t\to\pm\i} \|\Phi^t_H(\f,r) \| = \infty$$
for any $(\f,r)$ satisfying $r_4\neq 0$. 
\end{Main}

 \begin{Rem} We will see in Section \ref{sSmooth} that the construction of Theorem \ref{mC} can actually be carried out in any Gevrey class $G^\sigma$ with $\sigma>1$.  
 \end{Rem}
 
Notice that in the examples of Theorem \ref{mC}  the hyperplane $r_4=0$ is foliated by KAM tori with translation 
vector $\omega_0$, so the torus $\T^d\times\b{0}$ is not isolated. Theorem \ref{mC}  
gives however counter-examples for $d \geq 4$ to the positive measure accumulation by KAM-tori. Indeed, each point that lies outside this hyperplane diffuses to infinity 
along a sequence  of time. As we shall see in Proposition 4.3, its positive and negative  semi-orbits actually oscillate between $-\infty$ and $+\infty$ in projection to at least two action coordinates. 

It would be interesting to construct smooth examples with an isolated KAM-torus, thus showing that the phenomenon of Theorem \ref{mPrincipal} is purely analytic. On the other hand if Herman's conjecture is correct, 
then the phenomenon of Theorem \ref{mC} cannot be carried to the analytic setting. 
 
It is worth noting that Herman did also announce in \cite{H} the existence of counter-examples in the $C^\infty$ category 
to the positive measure conjecture, provided $d\geq 4$. However, he did not provide any clue to these examples and we are not aware whether the examples he had in mind had any invariant tori accumulating the KAM-torus. 

\medskip

\subsection{Plan of the paper}

The paper is organized in the following way. In section \ref{sBNF}
 we discuss the Birkhoff normal form and we give
a different (from the usual) characterization of it. In section \ref{s3} we formulate a KAM counter term theorem
which we use to give still another characterization of the Birkhoff normal form. Using this result we derive
Theorem \ref{mA} and \ref{mB} and R\"ussmann's theorem in sections \ref{sKAMstability} and  \ref{sCo-isotropic}. In section 
\ref{sSmooth} we prove Theorem \ref{mC}, 
and in section \ref{sProof} we give a proof of the KAM counter term theorem used in section \ref{s3}.

\bigskip

\subsection{Notations}

We denote by $\bD_{\delta}^d$ the polydisk in $\bC^d$ with center $0$ and radius $\de$. More generally if $d=(d_1,\dots,d_n)$ and 
$\de=(\de_1,\dots,\de_n)$, then
$$\bD^d_\de=\bD^{d_1}_{\de_1}\times\dots\times  \bD^{d_n}_{\de_n}.$$
Let $\T^d_{\rho}$ be the complex neighbourhood of width $\rho$ of
of $\bT^d$:
$$(\b{z\in\bC: |\Im z|<\rho}/\Z)^d.$$

A holomorphich function $f:\T^d_\rho\times \bD^e_\de\to\bC$ is {\it real} 
if it gives real values to real arguments. We denote by
$$\cC^{\o}(\bT^d_\rho\times\bD^e_\de)$$ 
the space of such real holomorphic functions which we provide with the norm
$$|f|_{\rho,\de}=\sup_{(\f,z)\in\bT^d_\rho\times\bD^e_\de}|f(\f,z)|.$$
We let
$$\cC^{\o}(\bT^d\times\b{0})=\bigcup_{\rho,\de}\cC^{\o}(\bT^d_\rho\times\bD^e_\de).$$

We denote by $\p_\f^\al f$ and $\p_z^\al f$  the partial derivatives
of $f$ with respect to $\f$ and $z$ respectively, with the usual 
multi-index notations. If $z=(z',z'')$ we say that
$$f\in \OO^j(z')$$
if and only if $\p_{z'}^{\al'} f(\f,0,z'')=0$ for all $|\al'|<j$.
We denote by $\p_\f f$ and $\p_z f$ the {\it gradient} of 
$f$ with respect to $\f$ and $z$, respectively, and by $\p_\f^2 f$ and 
$\p_z^2 f$ the corresponding {\it Hessian}.

For a function $f\in \cC^{\w}(\bT^d_\rho\times\b{0})$, 
$\cM(f)$ is the {\it mean value}
$$\int_{\bT^d}f(\f,z)d\f.$$

We shall also use the same notations for $\bC^n$-valued functions
$f=(f_1,\dots, f_n)$ with the absolute value replaced
by $|f|=\max_i |f_i|$ (or some other norm on $\C^n$).

\medskip

{\it Formal power series.} Let $z=(z_1,\dots,z_n)$. An element 
$$f\in \cC^{\o}(\bT^d_\rho)[[z]]$$
is a formal power series
$$f=f(\f,z)=\sum_{\al\in \bN^n} a_\al(\f) z^\al$$
whose coefficients $a_\al\in \cC^{\o}(\bT^d_\rho)$ 
(possibly vector valued). We denote by
$$[f]_j(\f,z)= \sum_{|\al|=j} a_\al(\f) z^\al,$$
the homogenous component of degre $j$, and
$$[f]^j=\sum_{i\le j}[f]_i.$$
The partial derivatives $\p_\f^\al f$ and $\p_z^\al f$
are well-defined  and if $z=(z',z'')$ we define that
$f\in \OO^j(z')$ in the same way as for functions.
The mean value $\cM(f)$ is the power series obtained by taking
the mean values of the coefficients. 

\medskip

{\it Parameters.} Let $B$ be an open subset of some euclidean
space. Define 
$$\cC^{\w,\i}( \T^d_\rho \times \bD^e_{\de},B)$$
 to be the set of $\cC^{\i}$ functions (possibly vector valued)
$$f:\T^d_\rho \times \bD^e_{\de}\times B\ni(\f,z,\o)\mapsto f(\f,z,\o)$$
such that for all $\o\in B$
\footnote{\ we appologize for the double use of $\o$}
$$f_\o:\T^d_\rho \times \bD^e_{\de}\ni (\f,z)\mapsto f(\f,z,\o)$$
is real holomorphic. We define
$$||f||_{\rho,\de,s}=\sup_{|\al|\le s}|\p^\al_\o f_\o|_{\rho,\de}.$$

\section{The Birkhoff Normal Form (BNF)}\label{sBNF}

Let
$$H(\f,r)=\l{\o_0,r}+\OO^2(r)\ \in\ \cC^{\w}(\bT^d_\rho\times\bD^d_\de)$$
and
$$\o_0\in DC(\k_0,\t_0).$$

\subsection{The Birkhoff normal form (BNF)}\label{s22}
Let us recall a well-known result.

\begin{Prop*}\label{bnfI}
There exist  
$$
\left\{\begin{array}{l}
f(\f,r)\in \cC^{\w}(\T^d_\rho)[[r]]\cap \OO^2(r)\\
N(r)\in \R[[r]]
\end{array}\right.$$ 
such that
$$H(\psi,r+\p_\psi f(\psi,r))=N(r).$$
Moreover, $N(r)$ is unique and $f$ is uniquely determined by fixing arbitrarily
the mean value $\cM(f)$.
\end{Prop*}

\begin{Rem} The unique series $N$ is the Birkhoff normal form
of $H$, denoted $N_H$. It is clear that
$$N_H(r)=\l{\o_0,r}+\OO^2(r).$$
We say that the unique $f$ for which $\MM(f)=0$ is the {\it generating function}
of the BNF, denoted $f_H$.
\end{Rem}

We know that the generating function $f_H$ is convergent if, and only if,  
$H$ is integrable \cite{I} (see also \cite{V,N}). It was known to Poincar\'e that 
for ``typical'' (in a sense we would call today generic) $H$, $f_H$ will be divergent. (Siegel 
\cite{S55} proved the same thing
in a neigbourhood of an elliptic equilibrium with another, and stronger,
notion of ``typical''.)

However, essentially nothing is known about the BNF itself when $\o_0$ is Diophantine. 
For example, the answers to the following questions are not known:
\begin{itemize}
\item[(i)] can $N_H$ be divergent?
\item[(ii)] if  $H$ is non integrable, can $N_H$ be convergent?
\end{itemize}
We only have a result of Perez-Marco \cite{P-M} saying that if 
the BNF $N_H$ is divergent for some $H$, then $N_H$ is divergent for 
``typical'' (i.e. except for a pluri-polar set) $H$. More generally, nothing 
is known about the set of all BNF's
$$\cB(\o_0)=\b{N_H: H(\f,r)=\l{\o_0,r}+\OO^2(r)}.
\footnote{\ apart from the fact that $N_H$ has some Gevrey-growth \cite{S05} and that $\cB(\o_0)$
contains all convergent series}
$$
Is it a ``large'' set or a ``small'' set in the space of all power series?
It has been shown \cite{B} that if $N_H$ fulfills a certain condition
$\cG$, which is prevalent in the space of power series, then
the invariant torus $\T^d\times\b{0}$ is {\it {   doubly} exponentially stable}.
However, it is not known whether $N_H$ can belong to 
$\cG$ when $H$ is non-integrable.

\subsection{Exact symplectic mappings and generating functions}\label{s23}

Consider the equations
\be\label{2.2}\left\{\begin{array}{l}
\f=\psi+p(\psi,r)\\
s=r+q(\psi,r)
\end{array}\right.\ee
with 
$$p,q\in \cC^{\w}(\bT^d\times \b{0 })$$
and 
$$\det(I+ \p_\psi p(\psi,r))\not=0$$
for all $(\psi,r)\in \bT^d\times \b{r\sim0 }$.

These equations can be solved  uniquely for
$(\psi,s)$ as
\be\label{2.3}\left\{\begin{array}{l} 
\psi=\f+\Phi(\f,r)\\
s= r+ R(\f,r)
\end{array}\right.\ee
with
$$\Phi,R\in \cC^{\w}(\bT^d\times \b{0 })$$
and
$$\det(I+ \p_\f\Phi(\f,r))\not=0$$
for all $(\f,r)\in \bT^d\times \b{r\sim 0 }$. Conversely,
the equations (\ref{2.3}), under the two supplementary conditions on $\Phi,R$,
can be solved uniquely for $(\f,s)$ as (\ref{2.2}), with the two
supplementary conditions on $p,q$.

\begin{Rem}
It is easy to verify that
$$p\in\OO(r)\quad{\textrm and}\quad q\in\OO^2(r)$$
if and only if
$$\Phi\in\OO(r)\quad{\textrm and}\quad R\in\OO^2(r).$$
\end{Rem}

The mapping
$$Z:(\f,r)\mapsto (\psi,s)$$
is a real analytic local diffeomeorphism
on $ \bT^d\times \b{r\sim 0 }$. It is 
{\it symplectic} if and only if the one-form
$$Z^*(rd\f)-(rd\f)$$ 
is closed, and it is {\it exact} if and only if this one-form is
exact.

\begin{Prop}\label{sympl-exact}
$Z$ is symplectic if and only if the one-form
$pdr+qd\psi$ is closed. $Z$ is exact if and only if the one-form
$pdr+qd\psi$ is exact 

If 
$$\Phi\in\OO(r)\quad{\textrm and}\quad R\in\OO^2(r),$$
then $Z$ is exact if and only if it is symplectic.
\end{Prop}

Hence, if $Z$ is exact there is a unique (modulo an
additive constant) function $f$ such that
$df=pdr+qd\psi$. The function $f$ is 
said to be a {\it generating function for} $Z$.

\begin{proof}
We have
$$
sd\psi-rd\f=(r+q(\psi,r))d\psi-rd\psi-   \p_\psi(rp)d\psi -
\sum_{i,j}r_j\p_{r_i}p_j(\psi,r)dr_i$$
and
$$d(rp)=\p_\psi(rp)d\psi+pdr+\sum_{i,j}r_j\p_{r_i}p_jdr_i.$$
Hence
$$
sd\psi-rd\f=qd\psi+pdr-d(rp)$$
which proves the first two statements.

Finally, if
$$\Phi\in\OO(r)\quad{\textrm and}\quad R\in\OO^2(r),$$
then
$$p\in\OO(r)\quad{\textrm and}\quad q\in\OO^2(r).$$
Now, $pdr+qd\psi$ is closed if and only if
for all $i,j$
$$\left\{\begin{array}{l}
\p_{r_i}p_j=\p_{r_j}p_i  \\
\p_{\psi_i}q_j=\p_{\psi_j}q_i \\ 
\p_{\psi_i}p_j=\p_{r_j}q_i
\end{array}\right.$$
By the symmetry condition on $\p_r p$ 
this implies that there exists a unique function $f(\psi,r)$
such that for all $j$
$$\p_{r_j}f=  p_j,\quad f(\psi,0)=0.$$
Then, for all $i,j$,
$$\p_{r_j}\p_{\psi_i}f=\p_{\psi_i}p_j=\p_{r_j}q_i$$
and, hence,
$$\p_{\psi_i}f(\psi,r)=q_i(\psi,r)+h_i(\psi).$$
Since $f,q\in \OO(r)$, this implies that $h_i=0$.
\end{proof}

\begin{Cor}\label{conjinv}
If
$$\begin{array}{rl}
Z:&\bT^d\times \b{r\sim 0 }\to \bT^d\times \b{r\sim 0 }\\
&(\f,r)\mapsto (\f+\Phi(\f,r),r+R(\f,r))
\end{array}$$
is a symplectic real analytic local diffeomorphism 
such that 
$$\Phi\in\OO(r)\quad{\textrm and}\quad R\in\OO^2(r),$$
then
$$N_{H\circ Z}=N_H.$$
\end{Cor}

\begin{proof}
Applying the BNF proposition of section \ref{s22} to $H$ and $\tilde H=H\circ Z$ we find two
generating  functions
$$f,\tilde f\in \cC^{\w}(\T^d_\rho)[[r]] \cap\OO^2(r).$$
By truncating these functions at degree $n$ and applying Proposition \ref{sympl-exact} we find 
two exact symplectic mappings $W_n$ and $\tilde W_n$ such that  
$$
H\circ W_n(\f,r) = N_H^n+\OO^{n+1}(r)$$
and
$$H\circ Z\circ \tilde W_n(\f,r) = N_{H\circ Z}^n+\OO^{n+1}(r)$$
By Proposition \ref{sympl-exact} again $Z\circ \tilde W_n$ has a generating function
$$g_n\in \cC^{\w}(\bT^d_{\rho''}\times \bD^d_{\de''})\cap\OO^2(r).$$
Letting $n\to\i$, the result now follows from the uniqueness of the 
BNF proposition of section \ref{s22}.
\end{proof}

\subsection{Another characterization of the BNF}\label{s24}

Let $P(r,c)$ be a power series in $r,c \in \R^d \times \R^d$. We say that
$$P(r,c)=0 \mod \OO^2(r-c)\quad\textrm{or}\quad
P(r,c)\in\OO^2(r-c)$$
if
$$P(r,c)=\l{r-c,Q(r,c)(r-c)}$$
for some matrix valued power series $Q(r,c)$. 
Using this notation any $P(r,c)$ can be written
$$P(c,c)+\l{\p_r P(c,c),r-c)} +\OO^2(r-c).$$

\medskip

\begin{Prop}\label{bnfII} Assume $\omega_0$ is a Diophantine vector. There exist
$$\left\{
\begin{array}{l}
f(\f,r,c)\in \cC^{\w}(\T^d_\rho)[[r,c]]\cap \OO^2(r,c)\\
\Omega(c)\in \bR^d[[c]]\\
\G(c)\in \bR[[c]]
\end{array}\right.$$ 
such that
\begin{equation}H(\psi,r+\p_\psi f(\psi,r,c))=\G(c)+ \label{BNFc} \end{equation}
$$\l{\Omega(c),r-c}+
\l{(r-c),F(\psi+\p_r f(\psi,r,c),r,c)(r-c)} .$$

 Moreover, if $\o_0$ is just supposed to have  rationally independent coordinates  (without the Diophantine assumption), we have that if 
\begin{equation}H(\psi,r+\p_\psi f(\psi,r,c))=\G(c)+ \label{BNFcq} \end{equation}
$$\l{\Omega(c),r-c}+
\l{(r-c),F(\psi+\p_r f(\psi,r,c),r,c)(r-c)}  { +\cO^{q+1}(c)}$$
for some $q \geq 1$, then
$\G(c)$ and $\Omega(c)$ are unique mod $\cO^{q}(c)$ and
$$\G(c)=N_H(c)+\cO^{q+1}(c)$$
and 
$$\Omega(c)=\p_c N_H(c)+\cO^{q}(c).$$ 
\end{Prop}

\begin{proof}
We must show not only that there exists at least one solution $f,\G,\Omega$ of this problem, but we must
also show that $\G,\Omega$ are the same for all such solutions. Let 
$$ H_j(\f,r)=[H(\f,r)]_j$$
be the homogeneous component of degree $j$
of $H(\f,r)$ and define $f_j(\psi,r,c)$, $\G_j(c)$,
$\Omega_j(c)$ and $F_j(\f,r,c)$ similarly.

{  For $j=1$, the equation becomes
$$\G_1(c)+\l{\Om_0,r-c}=\l{\om_0,r}$$
which gives $\Om_0=\om_0$ and $\G_1=\l{\om_0,c}$.
}

For $j=2$, the equation becomes
$$
\begin{array}{c}
\l{\o_0,\p_\psi f_2(\psi,r,c)}+H_2(\psi,r)=\G_2(c) + \qquad\\
\qquad+\l{\Omega_1(c),r-c}+\l{r-c,F_0(\psi)(r-c)}
\end{array}.$$
Write $H_2(\psi,r)$
$$=H_2(\psi,c)+\l{\p_r H_2(\psi,c),r-c}+\l{r-c,Q(\psi)(r-c)}.$$
Then we must have
$$\G_2(c)+\l{\Omega_1(c),r-c}=\MM(H_2(\cdot,c)+\l{\p_r H_2(\cdot,c),r-c})$$
which determines $\G_2$ and $\Omega_1$ uniquely.

If we take $F_0=Q$, then we get the equation for $f_2$:  
\begin{equation} \label{cf2} \<\omega_{0},\partial_{\psi}f_{2}(\psi,r,c)\>
=\\-{\cal V}\biggl(H_{2}(\psi,c)+\<\partial_{r}H_{2}(\psi,c),r-c\>\biggr)
\end{equation} 
where ${\cal V}=id-\MM$. Clearly this equation defines $f_2$ uniquely 
modulo a mean value $g_2$. But we can also add any term of degre two in $\OO^2(r-c)$ to
$f_2$ and still get a solution simply by changing the definition
of $F_0$. Hence $f_2$ is unique modulo a mean value $g_2$ and modulo 
$\OO^2(r-c)$. (In the sequel we must show, in particular, that the higher order terms
of $\Gamma$ and $\Omega$ remain the same for these different choices of $f_2$.)

We now proceed by induction on $j\geq 3$:  assume that we have  constructed for $2\leq m\leq j-1$, 
the homogeneous components   
$f_m(\psi,r,c)$, $\G_m(c)$, $\Omega_{m-1}(c)$ and $F_{m-2}(\f,r,c)$ 
and assume that $f_{m}(\psi,r,c)$ is unique  modulo a meanvalue 
$g_m(r,c)$ and modulo  $\OO^2(r-c)$ -- we have seen that this induction 
assumption is true for $j=2$.

For $j\ge 3$, the equation becomes
$$
\begin{array}{c}
\l{\o_0,\p_\psi f_j(\psi,r,c)}+ G_j(\psi,r,c)
=\G_j(c)+\l{\Omega_{j-1}(c),r-c}+\qquad \\
\qquad + \l{r-c,(K_{j-2}+F_{j-2})(\psi,r,c)(r-c)}
\end{array}$$
where $G_j(\psi,r,c)$
$$=[(H_2+\dots+H_{j})(\psi,r+\p_\psi f_2(\psi,r,c)+\dots + \p_\psi f_{j-1}(\psi,r,c),c)]_j$$
and
$K_{j-2}(\psi,r,c)$
$$=[(F_0+\dots+F_{j-3})(\psi+\p_r f_2(\psi,r,c)+\dots + \p_r f_{j-1}(\psi,r,c),c)]_{j-2}.$$

We write $G_j(\psi,r,c)$
$$=G_j(\psi,c,c)+\l{\p_r G_j(\psi,c,c),r-c}+\l{(r-c),Q(\psi,r,c)(r-c)}$$
and notice that $G_j(\psi,c,c)+\l{\p_r G_j(\psi,c,c),r-c}$ only depends on $f_2,\dots,f_{j-1}$
modulo their meanvalues and modulo $\OO^2(r-c)$ 
 -- hence this term is uniquely determined by  $H_2+\dots+H_j$.
Then
$$\G_j(c)+\l{\Omega_{j-1}(c),r-c}=\MM(G_j(\cdot,c,c)+\l{\p_r G_j(\cdot,c,c),r-c})$$
which determines $\G_j$ and $\Omega_{j-1}$ uniquely.

If we take $F_{j-2}=Q-K_{j-2}$, then we get for $f_j$ the equation
\be\<\omega_{0},\partial_{\psi}f_{j}(\psi,r,c)\>
=\\-{\cal V}\biggl(G_{j}(\psi,c,c)+\<\partial_{r}G_{j}(\psi,c,c),r-c\>\biggr)  \label{cfn} \ee 
which has a unique solution modulo a mean value $g_j(r,c)$. 
But we can also add any term of degre $j$ in $\OO^2(r-c)$ to
$f_j$ and still get a solution simply by changing the definition
of $F_{j-2}$.

{ This shows the existence of $f,\G$ and $\Om$ verifying
(\ref{BNFc}) up to any order $q$, as well as the uniqueness.

By Propositions \ref{bnfI} there exists
$$f\in\CC^{\w}(\bT^d)[[r]]\cap \OO^2(r)$$
such that
$$H(\psi, r+\p_\psi f(\f,r))=N_H(r).$$
Now
$$N_H(r)=N_H(c)+\l{\p_r N_H(c),r-c}+\OO^2(r-c)$$
and the uniqueness of $\Omega(c)$, mod $\cO^q(c)$,
gives the final statement.
}
\end{proof}

\section{A KAM counter term theorem and the BNF}\label{s3}

Let $B$ be the unit ball  centered at $\o_0$ or, more generally, the intersection of this
unit ball with an affine subspace of  $\R^d$ through $\o_0$.

Let $\k>0$ and $\tau>d-1$ be given numbers.

Let   $l:\bR  \to \bR$ denote a fixed non-negative $C^\infty$ function such that 
$|l| \leq 1$, and  $l(x)=0$ if $|x|\geq 1/2$ and 
$l(x) = 1$ if $|x|\leq 1/4$. 

\subsection{A cut-off operator and flat functions}\label{s32}

For $f  \in \cC^{\w,\i}( \T^d_\rho \times \bD^e_{\de},B)$, let
$$\PP(f)(\f,z,\omega)=  \sum_{n\in\Z^d\sm\{0\}} 
\hat f(n,z,\omega) e^{2\pi i \<n,\f\>} l(\lan n,\w \ran \frac{|n|^{\tau}}{\kappa} ).$$

\begin{Rem}Notice that $\PP(f)$ depends on the choice of $l,\t$ and $\k$. 
We shall not care about the dependence 
on the first two factors -- all constants will depend on $l$ and $\t$  
-- but we shall keep careful track on the dependence on $\k$.
\end{Rem}

Notice also that $g=\PP(f)$ is a flat function on $DC(\k,\t)$, i.e.
$$\p^\al_\f \p^\beta_z \p^\ga_\o g(\f,z,\o)=0$$
for all multi-indices $\a,\beta,\ga$ whenever $\o\in DC(\k,\t)$  --  a
function $g$ with this property is said to be $(\k,\t)${\it -flat}. In particular, if
$f=\PP(f)$, then $f$ is $(\k,\t)$-flat. 

\begin{Lem}\label{flat} We have
$$\|\PP(f)\|_{\rho',\delta,s}\leq C_{s}(\frac1\kappa)^s(\frac1{\rho-\rho'})^{(s+1)\t+d}
\|f\|_{\rho,\delta,s}$$
for any $\rho'<\rho$ and any $s\in\N$. The constant $C_s$ only depends, besides $s$, on $\t$
and l.
\end{Lem}

\begin{proof}
The Fourier coefficients (with respect to $\f$) verify
$$\aa{\hat f(n,\cdot,\cdot)}_{0,\de,s}\le \aa{f}_{\rho,\de,s} e^{-2\pi|n|\rho}.$$
The functions
$$l_n(\o)= l(\lan n,\w \ran \frac{|n|^{\tau}}{\kappa} )$$
verify
$$\aa{l_n}_{0,0,s}\le |n|^{(\t+1)s}\frac1{\k^s}\aa{l}_{0,0,s}.$$

Hence for $|\a|\leq s$ and $(\f,z,\omega)\in  \T^d_{\rho'}\times \bD^{d}_{\delta}\times B$
$$|\pa^\al_\o \PP(f)(\f,z,\omega)|\leq $$
$$C_{s} 
\sum_{n\ne 0}e^{2\pi |n|\rho'}\biggl(\|\hat f(n,\cdot,\cdot)\|_{0,\de,s}
+\|\hat f(n,\cdot,\cdot)\|_{0,\de, 0}|n|^{(\t+1)s}\frac1{\k^s}\biggr)$$
which gives the estimate. (Here we have used  Proposition
\ref{hadamard}.)
\end{proof}

\subsection{A counter term theorem}\label{sCounterthm}

\begin{Prop}\label{counterthm}
Given $0<\k<1$ and $\t>d-1$. Then, for all $s\in\N$, there exist 
{ non-negative
constants   (only depending on $s$ and $\t$) 
$$\al(s)\ge (s-t)+\al(t),\quad s\ge t\ge0,$$
}
such that if 
$$H(\f,r)= N^q(r)+\OO^{q+1}(r)\in\CC^{\w}(\T^d\times \b{0}),\quad q\ge \al(1)+1,$$
with 
$$N^q(r)=\l{\o_0,r}+\OO^2(r),$$
then there exist $\rho,\de>0$ and 
$$
\left\{\begin{array}{l}
f=f(\f,r,c,\o)\in\CC^{\w,\i}(\T^d_\rho\times \bD^{d}_{\de}\times \bD^{d}_{\de},B)\cap \OO^2(r,c)\\
\Lambda=\Lambda(c,\o)\in \CC^{\w,\i}(\bD^{d}_{\de},B)
\end{array}\right.$$
such that 
\begin{multline} \label{p35}
H(\psi, r+\p_\psi f(\psi,r,c,\o))+\l{\o+\Lambda(c,\o), r+\p_\psi f(\psi,r,c,\o)}\\
=\l{\o,r-c}+\OO^2(r-c) +g\end{multline}
(modulo an additive constant that depends on $c,\o$) with $g$ $(\k,\t)${\it -flat} and 
{ $g\in \cO^2(r,c)\cap \cO^q(c)$}.

Moreover,
\begin{itemize}
\item[(i)]
there exist constants $C_s$, only depending on $s,H,l,\t$ such that
$$\aa{\Lambda+\p_rN^q}_{0,\eta,s}+\aa{f}_{\rho,\eta,s}\le C_s\eta^{q}
(\frac{1}{\k\eta})^{\al(s)}$$
for any $\eta<\de$ 
\item[(ii)] there exists a constant $C$, only depending on $H,l,\t$,
such that
$$\de \ge \frac1C\k^{\frac{\al(1)}{q-\al(1)}}$$
\item[(iii)] if
$$\o_0\in DC(2\k,\t)$$
 then the mapping
$$\bD^{d+1}_{\de'}\ni (c,\la)\mapsto \Lambda(c,(1+\la)\o_0)\in \bC^d$$
is real holomorphic for some $\de'$.
\end{itemize}
\end{Prop}

\begin{Rem}
Notice that this proposition (except part (iii)) does not require that $\o_0$ is Diophantine.
\end{Rem}

We shall prove this proposition in section 5, but here we shall derive its consequences.

\begin{Cor}\label{bnfIII} Given $0<\k<1$ and $\t>d-1$ and 
{ non-negative constants  $\al(s)$
as in Proposition  \ref{bnfII}. 
}

If 
$$H(\f,r)= N^q(r)+\OO^{q+1}(r)\in\CC^{\w}(\T^d\times \b{0}),\quad q\ge \al(1)+1,$$
with 
$$N^q(r)=\l{\o_0,r}+\OO^2(r),$$
then there exists a unique $\CC^{\i}$ function $\Omega: \R^d \to \R^d$,
defined in a neighborhood of $0$ given by 
{
$$|c|< \eta_0=\frac1{C'}\k^{\frac{\al(1)}{q-\al(1)}},$$
}

where $C'$ only depends on $H,\t,l$, such that 
$$\Omega(c)+\Lambda(c,\Omega(c))=0.$$
Moreover,
\begin{itemize}

\item[(i)]  for any $s\in\bN$ there exists a constant $C'_s$ such that
$$\aa{\Omega-\p_rN^q}_{\CC^s(|c|<\eta)}\le 
C'_s\eta^q(\frac{1}{\k\eta})^{\al(s)}$$
for any $\eta<\eta_0$ 

{ \item[(ii)] the Taylor series of $\Omega$ up to degree $q-1$ at $c=0$ is given by $\p_r N^q(c)$.}

\item[(iii)] If $\omega_0 \in DC(\k,\t)$, the Taylor series of $\Omega$ at $c=0$ is given by
$\p_r N_H(c)$.
\end{itemize}
\end{Cor}

\begin{Rem}
This corollary gives a third characterization of the BNF.
\end{Rem}

\begin{proof}
We have that $\o_0+\Lambda(0,\o_0)=0$ (because $f,g \in\OO^2(r,c)$) and  by (i) of Proposition \ref{counterthm}   
$$|\p_\o \Lambda(c,\o)| \le 
C_1\eta_0^{q}(\frac{1}{\k\eta_0})^{\al(1)}\lsim \frac12$$
for $|c|<\eta_0$ and $\o\in B$. The local existence of 
$\Omega$ follows now by the implicit function theorem.
By a Cauchy estimate and  (i) of Proposition \ref{counterthm}   
$$|\p_c \Lambda(c,\o)| \lsim
C_0\eta_0^{q-1}(\frac{1}{\k\eta_0})^{\al(0)}+ \aa{\p_c^2 N^q}_{0,\eta_0,0}\lsim  \bar C$$
where $\bar C$ only depends on $H,\t,l$,
which implies that $\Omega$ is defined for $|c|<\eta_0$, provided $C'$ is sufficiently large depending only on  $H,\t,l$.

{ Now since $g=\cO^q(c)$, \eqref{p35} yields
$$H(\psi, r+\p_\psi f(\psi,r,c,\Omega(c)))
=
\G(c,\o)+\l{\Omega(c),r-c}+\OO^2(r-c)+\cO^q(c)$$
and  we get (ii) from
the uniqueness up to $\cO^{q-1}(c)$ of $\Omega$ as seen  in
{ Proposition \ref{bnfII}}.

If $\o_0\in DC(\k,\t)$, then $(\kappa,\tau)$-flatness of $g$ in \eqref{p35} implies 
$$H(\psi, r+\p_\psi f(\psi,r,c,\Omega(c)))
=
\G(c,\o)+\l{\Omega(c),r-c}+\OO^2(r-c)$$
which by Taylor expansion at $c=0$ and the uniqueness of $\Omega$ in Proposition  \ref{bnfII} implies (iii).}

It remains to prove the estimates (i). If we define
$$
F(c,\tilde \o)=\Lambda(c,\tilde \o+\p_r N^q(c))+\p_r N^q(c)$$
and $\tilde\Om(c)=\Om(c)-\p_r N^q(c)$, then
$$\tilde\Om(c)+F(c,\tilde \Om(c))=0.$$
Now
$$|\p_{\tilde \o} F(c,\tilde\o)| \lsim \frac12,$$
and
$$\aa{F}_{\CC^s}\le \tilde C_s\eta^q(\frac{1}{\k\eta})^{\al(s)},$$
where the $\cC^s$-norm is taken over all $|c|<\eta,\  |\tilde\o|<\frac12$.

Then, by an induction,
$$\aa{\tilde\Om}_{\CC^s}\lsim C_s' \aa{F}_{\CC^s}.$$
\end{proof}

It follows immediately under the same hypothesis as in Corollary \ref{bnfIII}

\begin{Cor}\label{kamtori} 
If $\Omega(c)\in DC(\k,\t)$, then
$$H\circ Z_c(\f,r)=\G(c)+\l{\Omega(c),r-c}+\OO^2(r-c)$$
where $Z_c$ is the exact symplectic mapping generated by 
$f(\f,r,c,\Omega(c))$. 
Moreover
$$(\f,c)\mapsto Z_c(\f,c)$$
is a local diffeomorphism.
\end{Cor}


\section{Nondegenerate  BNF and KAM stability}\label{sKAMstability}

This section is devoted to the proof of Theorem \ref{mB}.

\subsection{Transversality} \label{s35}

\begin{Lem}\label{transversality}
If $N_H(r)$ is  non-degenerate, then
there exist $p,\s>0 $ such that
for any $k\in\Z^d \sm \{0\}$ there exists a unit vector
$u_k\in\bR^d$ such that the series
$$
f_k(r)=\l{\frac{k}{|k|},\p_r N_H(r)}$$
is $(p,\s)${\it-transverse in direction} $u_k$, i.e.
$$\max_{0\le j\le p}
|\p^j_{t} f_k(tu_k)_{|t=0}|\ge \s.$$
\end{Lem}

\begin{proof}
Indeed, if this were not true, there would exist a sequence $k_{n}\in\Z^d\setminus\{0\}$ 
such that for any $u\in\R^d$
$$\max_{0\le j\le n}|\p^j_{t} f_{k_n}(tu)_{| t=0}|<\frac1n.$$
Extracting a subsequence for which $k_{n_{j}}/|k_{n_{j}}|\to v\in\R^d$ clearly gives 
that $\l{v,\p_r N_H(r)}=0$, i.e. $N_H$ would be degenerate.
\end{proof}

Consider now these $p,\s$.
Let $\Omega\in \CC^p(\b{|c|<\eta})$
and assume
$$\aa{\Omega-[\p_r N_H]^p}_{\CC^p(\b{|c|<\eta})}\le \frac\s2.$$

\begin{Lem}\label{pyartly}
If $N_H$ is $(p,\s)$-transverse (in some direction), then
$$\Leb\b{|c|<\eta: |\l{\frac{k}{|k|},\Omega(c)}|<\ep}
\le C_p(\frac\ep\s)^{\frac1p}\eta^{d-1}$$
for any { $\eta, k,\eps$}.
\end{Lem}

\begin{proof}
We have, for some $0\le j\le p$, 
$$
|\p^j_{t}  \l{\frac{k}{|k|},\Omega(c+tu)}|\ge \frac\s2$$
for all $|c+tu|<\eta$. The estimate is now an easy calculation.
\end{proof}

\subsection{Proof of Theorem \ref{mB}}\label{s36} 

By Lemma \ref{transversality} we are given $p$ and $\s$ that correspond 
to the transversality of the formal series $N_H$. 
{ We can assume without restriction that $\s\le 1$.}
Fix $q=(1+2p)\a(p)+1$. Performing a conjugacy, we  
can assume without restriction that $H$ is given by its Birkhoff normal form power series up to order $q$ plus higher order terms:
$$H(\f,r)=N^q(r)+\OO^{q+1}(r)$$

We shall apply Proposition \ref{counterthm} and Corollaries \ref{bnfIII}--\ref{kamtori}
with 
$$\t=dp+1 \quad \textrm{and}\quad 0<\k\le\s^q\le 1.$$
{Now let 
$$\eta:=\frac1{C''}(\frac\k\s)^{\frac1{2p}}.$$
Since $q\ge (1+2p)\al(1)+1$ we have $\eta\le \eta_0$ 
for all $C'' \geq C'$, with $\eta_0$ and $C'$ defined in Corollary \ref{bnfIII}.  
Then $\Omega$ is defined in $\b{|c|<\eta}$ and
\begin{equation}\aa{\Omega-[\p_r N_H]^p}_{\CC^p(\b{|c|<\eta})}\le C'_p\eta^q(\frac{1}{\k\eta})^{\al(p)}+
\aa{[\p_r N_H]^p-\p_r N_H^q}_{\CC^p(\b{|c|<\eta})}
\end{equation}
which is $$\leq \tilde{C}\eta$$ 
since $q\ge (1+2p)\al(p)+1$ --  notice that $\tilde C$ is independent of $C'' \geq C'$.
Finally if $C''$ is sufficiently large (depending on $p,\t$,$H$,$l$, thus on $q$) 
we have that $\tilde{C}\eta\leq \s/2$.
}

By Lemma \ref{pyartly}
$$ \Leb\b{|c|<\eta: |\l{\frac{k}{|k|},\Omega(c)}|<\ep} \lsim (\frac\ep\s)^{\frac1p}\eta^{d-1},$$
hence 
\begin{align*} \Leb\b{|c|<\eta: \Omega(c)\notin DC(\k,\t)}&\lsim (\frac{\k}{\s})^{\frac1p}\eta^{d-1}\\
&\lsim \eta\Leb\b{|c|<\eta} \end{align*}
provided $\kappa$ is sufficiently small. Hence, the set
$$\b{|c|<\eta: \Omega(c)\in DC(\k,\t)}$$
is of positive measure and density $1$ at $0$ as $\k\to0$.
Theorem \ref{mB} now follows from Corollary \ref{kamtori}.

\section{Analytic KAM tori are never isolated. Degenerate BNF and Invariant co-isotropic submanifolds. } \label{sCo-isotropic} 

This section is devoted to the proof of Theorem \ref{mA} and of R\"ussmann's theorem.

Let $q=\a(1)+1$ and assume, after a conjugacy, that  
$$H(\f,r)=N^q(r)+\OO^{q+1}(r)$$
We shall apply Proposition \ref{counterthm} and Corollaries \ref{bnfIII}+\ref{kamtori}
with 
$$q=\al(1)+1,\quad
\t=\t_0\quad \textrm{and}\quad \k=\frac{\k_0}2.$$

Then
$$\Omega(c)+\Lambda(c,\Omega(c))=0$$
and
$$\Omega(c)=\p_rN_H(c)+\OO^{\i}(c).$$

Since $N_H$ is $j$-degenerate
we have
$$\p_v^nN_H(0)=0\qquad \forall n\ge0$$
for any $v\in\Lin(\g=(\g_1,\dots,\g_j))$,
where $\p_v$ is the directional 
derivative in direction $v$. From this we derive that
$$\p_v^n(\o_0+\L(\cdot,\o_0))_{|c=0}=0\qquad \forall n\ge0.$$
Since $s\mapsto \L(\l{s,\g},\o_0)$ is an analytic function in $s \in \R^j$, $s\sim 0$,
it must be identically $0$, hence 
$\Omega(\l{s,\g})$ is identically $\o_0$,
i.e.
$$\Omega(\l{s,\g})\in DC(\k,\t)$$
for all sufficiently small $s$. 

From Corollary \ref{kamtori}
it follows  that for any $c\in \Lin(\g)$ sufficiently small
$H$ has a KAM-torus with frequency $\o_0$ and that
the set of all these tori,
$$\bigcup_{c\in\Lin(\g)}Z_c(\T^d,c),$$
is a $(d+j)$-dimensional subvariety.
This completes the proof of Theorem \ref{mA}.

\medskip

When $N_H$ is $(d-1)$-degenerate, then
$$\p_r N_H(c)=\m(\l{c,\o_0})\o_0$$
where $\m(t)=1+\OO(t)$ is a formal power series in one variable.

Since 
$$\m(\l{c,\o_0}) \o_0+\Lambda(c,\m(\l{c,\o_0})\o_0)=\OO^{\i}(c),$$
taking $c=t\o_0$, we have (assuming $\o_0$ is a unit vector) 
\be\label{3.6}\m(t)\o_0+\Lambda(t\o_0,\m(t)\o_0)=0\ee
modulo a term in $\OO^{\i}(t)$.
Since, by Proposition \ref{counterthm} (iii), 
the lefthand side is analytic in $t\o_0$ and
$\mu$ we obtain from any of the equations (\ref{3.6})  that
$\mu(t)$ is a convergent power series. Then
$$t\mapsto \m(t)\o_0+\Lambda(t\o_0,\m(t)\o_0)$$
is analytic for $t\sim 0$, hence identically zero.
We derive from this that
$$\Omega(c)= \m(\l{c,\o_0})\o_0,$$
i.e.
$$\Omega(c)\in DC(\k,\t)$$
for all sufficiently small $c$. 
R\"ussmann's theorem now follows from Corollary \ref{kamtori} as in the proof
of Theorem \ref{mA}.

\section{Smooth non KAM stable Diophantine tori}\label{sSmooth}

\subsection{The smooth case in $d=2$ degrees of freedom. }  \label{secd2}
We let $H$ be as in the introduction but we only assume that $H$ is of class $C^\infty$. 
The results of sections \ref{sBNF} and  \ref{sKAMstability} remain valid but we will only have $C^\infty$ instead of analytic functions. For example we will not be able to use the analyticity dependance of $\Lambda$ on the first variable, that is crucial in the degenerate situation  as shown in Section \ref{sCo-isotropic}.  

But let us examine the frequency function $\Omega(c)$ 
given by Corollary \ref{bnfIII}. It is a smooth function from a neighborhood of $0$ in $\R^2$ to $\R^2$, such that $\Omega(0)=\omega_0 \in  DC(\k_0,\t_0)$. We restrict to a neighborhood where $ \omega_{0,i}/2 \leq \Omega_i(c) \leq 2\omega_{0,i}$. A vector $\Omega(c)=(\w_1,\w_2) \in \R^2$ then satisfies a Diophantine condition for flows as in (\ref{cd}) as soon as $\a(c)=\w_1/\w_2$ satisfies a  Diophantine condition for diffeomorphisms, of the form $|k\a+l| \geq C \k_0/ |k|^{\tau_0}$, with $C$ some constant that only depends on $\omega_0$. 
 
 If we restrict $\a(\cdot)$ to any segment $I$ that goes through $0$ we get a smooth real function such that $\a(0)$ satisfies the latter Diophantine condition for diffeomorphisms. As explained in Proposition 3 of \cite{FK}, the one dimensional phenomenon here is that, provided $\k_0$ and $\tau_0$ are 
relaxed to $\k<\k_0/2$ and $\tau =\tau_0+1$, then for a positive measure set of points in $I$, $\a$ satisfies a Diophantine condition for diffeomorphisms.  Indeed $\a(0)$ is a density point in $DC(C\k,\tau)$ and the alternative for $\a$ are {\it (i) : $\a$ is locally constant $\a\sim \a_0 \in DC(C\k,\tau)$ on a neighborhood of $0$ in $I$, or (ii) : $\a$ is not locally constant and it takes a positive measure set of values in  $DC(C\k,\tau)$ on a positive measure set of points in $I$.} 
 
  We conclude that for a positive measure set of $c$ in any neighborhood of $0$, $\Omega(c) \in  DC(\k,\t)$, so that Corollary \ref{kamtori} yields the following result, that can be coined {\it Herman's last geometric theorem} since it is just the flow version of the disc diffeomorphisms theorem treated in \cite{FK}.

\begin{Principal}  Let $H \in C^\infty(\T^2\times\R^2)$ and assume that  $\bT^2\times\b{0}$ is a KAM torus. Then  $\bT^2\times\b{0}$ is  accumulated by a positive measure set of 
smooth KAM tori with Diophantine translation vectors.
\end{Principal}

Note that we do not speak about KAM stability because the density one requirement does not hold in this 'twistless' situation.

\subsection{A smooth counter-example in $d\geq 4$ degrees of freedom.}

A vector $\a=(\a_1,\a_2) \in \R^2$ is said to be {\it Liouville}
if $(k,\a)=0 \implies k=(0,0)$ and if for any $N>0$ there exists $k \in \Z^2-\{0,0\}$ such that $|(k,\a)|<\|k\|^{-N}$.

We call a sequence of intervals (open or closed or halfopen) $I_n=(a_n,b_n) \subset ]0,\i[$ 
an increasing cover of the half line if :
\begin{enumerate} 
 \item  $\lim_{n \to -\infty} a_n =0$
 \item  $\lim_{n \to +\infty} a_n =+\infty$
\item $a_n<b_{n-1}<a_{n+1}<b_n$
\end{enumerate}

\begin{Prop} \label{colimacon} Let $(\omega_1,\omega_2,\omega_3) \in \R^3$ be fixed. 
For every $\eps>0$ and every $s\in \N$, there exist an increasing cover $(I_n)$ of $]0,\i[$ and functions $f_i \in C^\infty(\R,(0,1))$, $i=1,2,3$, such that $\|f_i\|_s<\eps$ and 
\begin{itemize} 

\item For each  $n\in \Z$, the functions $f_1$ and $f_2$ are  constant on $I_{3n}$ :  
$${f_1}_{| I_{3n}}\equiv \bar{f}_{1,n}, \quad {f_2}_{| I_{3n}}\equiv \bar{f}_{2,n}$$ 
\item For each  $n\in \Z$, the functions $f_1$ and $f_3$ are  constant on $I_{3n+1}$ :
$${f_1}_{| I_{3n+1}}\equiv \bar{f}_{1,n}, \quad {f_3}_{| I_{3n+1}}\equiv \bar{f}_{3,n}$$ 
\item For each  $n\in \Z$, the functions $f_2$ and $f_3$ are  constant on $I_{3n-1}$ :
$${f_2}_{| I_{3n-1}}\equiv \bar{f}_{2,n}, \quad {f_3}_{| I_{3n-1}}\equiv \bar{f}_{3,n-1}$$ 
\item The vectors $(\bar{f}_{1,n}+\omega_1,\bar{f}_{2,n}+\omega_2)$, $(\bar{f}_{1,n}+\omega_1,\bar{f}_{3,n}+\omega_3)$ and  $(\bar{f}_{2,n}+\omega_2,\bar{f}_{3,n}+\omega_3)$ are Liouville.
\end{itemize}
\end{Prop}

\begin{Rem}
It follows that  $f_1,f_2,f_3$ are $\cC^{\i}$-flat at zero.
\end{Rem}

\begin{proof} We want to construct $f_1(\cdot)$ such that $f_1$ is constant equal to $\bar{f}_{1,n}$ on $[a_{3n},b_{3n+1}]$ for every $n\in \Z$. 
The crucial observation in the construction of $f_1$ is that the segments $[a_{3n},b_{3n+1}]$ are mutually disjoint.

We will then construct similarly $f_2$ and $f_3$ and explain why the Liouville conditions can also be required in addition. 

Fix $\zeta \in C^\infty(\R,[0,1])$ be such that $\zeta(x)=0$ if $x\leq -1$ and $\zeta(x)=1$ if $x\geq 0$. 
Define a sequence $u_n>0$ such that  $a_{3n}-u_n > b_{3n-2}$ and $b_{3n+1}+u_n < a_{3n+3}$. 
Observe that 
$$g_n(x):= \zeta(u_n^{-1}(x-a_{3n}))-\zeta(u_n^{-1}(x-b_{3n+1}-u_n))$$
satisfies $g_n(x)=1$ if $x \in [a_{3n},b_{3n+1}]$ and $g_n(x)=0$ for $x> a_{3n+3}> b_{3n+1}+u_n$ and for $x<b_{3n-2}<a_{3n}-u_n$. Hence the function 
$$f_1=\sum_{n\in \Z} \bar{f}_{1,n} g_n$$
solves our problem and by just requiring the bound $(B_\eta)(n) : |\bar{f}_{1,n}|< \eta u_n^n $ for every $n$ and supposing that $\sum |u_n|<\infty$ we get that for any $s$ and any $\eps$ one can choose $\eta$ to guarantee that the resulting function $\|f_1\|_s<\eps/3$. We define the other functions similarly and then add the Liouville constraints without any problem since the condition $(B_\eta)(n)$ is  open.  \end{proof}

Given a cover $(I_n)$ as in Proposition \ref{colimacon}, we can define  another cover $(I'_n)$ such that $I'_n$ is strictly contained in $I_n$ for every $n$. 

\subsection{Proof of theorem \ref{mC}}

Define
$$H_0(\f,r)=\l{\omega_0,r}+f_1(r_4)r_1+f_2(r_4)r_2+f_3(r_4)r_3$$
where $f_1,f_2,f_3$ are as in Proposition \ref{colimacon} and $\omega_0=(\o_1,\o_2,\o_3,\o_4)$. 

Notice that as a consequence of Proposition \ref{colimacon} we have that on each $I_n$ two of the coordinates of $(f_1+\o_1,f_2+\o_2,f_3+\o_3)$ are constant and form a Liouville vector. We denote $\hat{I}_n=\T^4\times \R^3 \times I_n$. Let $\cH$ be the set of $H\in C^\infty(\T^4\times\R^4)$ such that $H$ does not depend on $\f_4$. For $H \in \cH$ the flow $\Phi_H^t$ leaves $r_4$ invariant.  We will show how to make arbitrarily small perturbations inside $\cH$ of $H_0$  on any $\hat{I}_n$ that create huge oscillations of the corresponding flow in two of the three directions $r_1,r_2,r_3$.   These perturbations will actually be compositions inside $H_0$ by exact symplectic maps obtained from  suitably chosen generating functions.
Iterating the argument gives a construction by successive conjugations scheme similar to \cite{AK}. The difference here is that the conjugations will be applied in a "diagonal" procedure to include more and more intervals $I_n$ into the scheme. Rather than following this diagonal scheme which would allow to define the conjugations explicitly at each step, we will actually adopt a $\GG^\delta$-type construction {\it \`a la Herman} (see \cite{FH}) that makes the proof much shorter and gives slightly more general results. 

Let $\cU$ be the set of exact symplectic diffeomorphisms $U$ of $\T^4\times \R^4$ such that $U(\f,r)=(\psi,s)$ satisfies $s_4=r_4$. In particular,  $U\in \cU$ implies that $U(\hat I_n)=\hat I_n$ for any $n \in \Z$.

\begin{Prop} \label{main.liouville} Let $I={I}_{n}$ for some $n$. For any $\eps>0, s \in \N, \Delta>0, A>0$ and any $V \in \cU$, there exist $U \in \cU$ and $T>0$ such that there exist $(i_1,i_2)\in \{1,2,3\}$,
distinct,  such that for $i=i_1$ and $i=i_2$ we have 
\begin{enumerate}
\item $U={\rm Id} { \ \rm on \ }  {\hat{I}}^c$
\item ${\|H_0\circ U \circ V -H_0\circ V \|}_s<\eps$
\item $\displaystyle { \sup_{0< t < T} | (\Phi_{H_0\circ U\circ V}^t(p))_{4+i_1}|>A,}$ for any $p\in \widehat{I'}$ such that $\|p\|\leq \Delta$
\item $\displaystyle  \sup_{0< t < T} | (\Phi_{H_0\circ U\circ V}^{-t}(p))_{4+i_2}|>A,$ for any $p\in \widehat{I'}$ such that $\|p\|\leq \Delta$
\end{enumerate}
\end{Prop}

\begin{proof} Since $V$ preserves $\hat I$ and since $\phi^t_{H_0\circ U\circ V}$ is conjugate to $\Phi^t_{H_0\circ U}$ it is sufficient to prove the proposition for $V={\rm Id}$. Indeed, given $V$ such that $V \hat I=\hat I$, and applying the Proposition with $V={\rm Id}$ and with constants $\eps'\ll \eps$ and $A'\gg A$ yields $2$ and $3$ including $V$.

Assume hereafter that $I=I_{3n}$, the other cases being exactly similar. Let $a \in C^\infty(\R)$ be such that 
$a(\xi)=0$ if $\xi \notin I$ and $a(\xi)=1$ if $\xi \in I'$ (remember that $I'$ is strictly included in $I$).  

Let $\bar{f}_1:={f_1}_{| I}, \bar{f}_2:={f_2}_{| I}$ and $\bar{F}_1=\bar{f}_1 +\omega_1$,  $\bar{F}_2=\bar{f}_2 +\omega_2$. Let $(q_1,q_2) \in \Z^2-\{0,0\}$ such that $|q_1|>A+\Delta$ and $|q_2|>A+\Delta$ and $|q_1\bar F_1+q_2 \bar F_2|<\eta \min(q_1^{-2s},q_2^{-2s})$ where $\eta= \eps /((2\pi)^{s+1} \|a\|_s)$. 

Define the following generating function $k\in \cH$ 
$$k(\psi,r)=a(r_4) \sin(2\pi (q_1\psi_1+q_2\psi_2))$$
and let $U=(\Phi,R)\in \cU$ be the symplectic diffeomorphism associated to $k$. Then
$R(\f,r)$ equals
$$(r_1+ 2\pi q_1 a(r_4) \cos(2\pi(q_1\f_1+q_2\f_2)),r_2+ 2\pi q_2 a(r_4) 
\cos (2\pi(q_1\f_1+q_2\f_2)),r_3,r_4)$$
so that $U={\rm Id} { \ \rm on \ }  {\hat{I}}^c$ and
$H_0\circ U(\f,r)$ equals 
$$H_0(r)+2\pi a(r_4) \left(q_1(f_1(r_4)+\omega_1) +q_2(f_2(r_4)+\omega_2) \right) 
\cos (2\pi(q_1\f_1+q_2\f_2)).$$
Hence $H_0\circ U(\f,r)- H_0(r)=h(r,\f)$ with $h\equiv 0$ if $r_4 \notin I$ and if $r_4\in I$ we have 
that  $h(r,\f)=2\pi a(r_4) (q_1\bar F_1+q_2 \bar F_2)  \cos (2\pi(q_1\f_1+q_2\f_2))$
thus the required $\|h\|_s<\eps$. 

On the other hand we have that on $\hat I$ the flow $\Phi_{H_0}^t$ is completely integrable with tori 
$\cT_r= \{r\}\times \T^4$ 
carrying the frequencies $(\bar{F}_1,\bar F_2,F_3(r_4),\omega_4)$. Recall that $\bar{F}_1$ and $\bar F_2$ are independent over $\Z$, that is, the dynamics of the translation flow $T^t_{\bar F_1,\bar F_2}$ is minimal.  But under the change of variable $U$ the torus $\cT_r$ for $r_4 \in I'$  becomes $\cT'_r=\{(r_1-2\pi q_1\cos (2\pi(q_1\f_1+q_2\f_2)), r_2-2\pi q_2 \cos (2\pi(q_1\f_1+q_2\f_2)),r_3,r_4)  :   (\f_1,\ldots,\f_4) \in  \T^4 \}$. Also, the change of variable is such that $(\Phi(\f,r))_j=\f_j$ for $j=1,2,3$. All this implies the third claim of Proposition \ref{main.liouville} since we took  $|q_1|>A+\Delta$ and $|q_2|>A+\Delta$.  \end{proof}

It is easy now to deduce Theorem \ref{mC} and in fact a stronger version of it.  Define for this purpose $\cU_0$ the subset of $U \in \cU$ such that $U-{\rm Id}=\cO^\infty(r_4)$ and $\cH_0$ to be the set of hamiltonians of the form $H_0\circ U, U\in \cU_0$. Finally we denote $\bar{\cH_0}$ the closure in the $C^\infty$ topology of $\cH_0$.

\begin{Prop} \label{theo.gdelta} Let $\cD$ be the set of hamiltonians $H \in \bar{\cH_0}$ such that 
\begin{equation} \label{eq.diffuse.prop} \limsup \|\Phi^t_{H} (p) \| = \infty \end{equation}
for any $p=(\f,r)$ satisfying $r_4\neq 0$. More precisely, for each $p$ such that $p_8\neq 0$ we have that there exist $(i_1,i_2)\in \{1,2,3\}$,distinct, such that for $i=i_1$ and $i=i_2$ it holds that 
\begin{equation} \label{eq.diffuse2} \limsup_{t \to \pm \infty } (\phi^t_{H}(p))_{4+i_1} = +\infty, \quad \liminf_{t \to \pm \infty} (\phi^t_{H} (p))_{4+i_2} = -\infty  \end{equation}

Then $\cD$ is  a dense (in the $C^\infty$ topology) $\GG^\delta$ subset of $\bar{\cH_0}$ 
\end{Prop}

\begin{proof} For $n,\Delta,A,T \in \N^*$ and $1\leq i_1< i_2 \leq 3$ let
$ \cD(n,\Delta,A,T,i_1,i_2)$ be the set
$$\left\{  H \in  \bar{\cH_0}  :   
 \sup_{0< t < T} \min_{i=i_1,i_2; j=1,-1}  \min_{p\in \widehat{I'_n} \cap \{\|p\|\leq \Delta\} }| (\phi_{H}^{jt}(p))_{4+i}|>A  \right\}.$$

It is clear that $\cD(n,\Delta,A,T,i_1,i_2)$ are open subsets of  $ \bar{\cH_0}  $ in any $C^s$ topology. On the other hand we have that 
$$\cD = \bigcap_{A \in \N^*} \bigcap_{n \in \N^*} \bigcap_{\Delta \in \N^*}  \bigcup_{T \in \N^*}  \bigcup_{(i_1,i_2)  \in \{(1,2),(1,3),(2,3)\} } \cD(n,\Delta,A,T,i_1,i_2)$$
but Proposition \ref{main.liouville} precisely states that 
$$ \bigcup_{T \in \N^*}  \bigcup_{(i_1,i_2)  \in \{(1,2),(1,3),(2,3)\} } \cD(n,\Delta,A,T,i_1,i_2)$$
 is dense in  $ \bar{\cH_0}  $ in any $C^s$ topology, which ends the proof of the theorem. 
 
 \end{proof}

The same result of Proposition \ref{theo.gdelta}  holds in any Gevrey class $G^\sigma$, for any $\sigma>1$. The proof of the latter fact  follows exactly the same line as the $C^\infty$ case with the following simple modifications. 

\begin{itemize}
\item[-] The compactly supported function $\zeta$ of Proposition \ref{colimacon} is taken to be in $G^\sigma$, as well as the function $a$ in the proof of Proposition \ref{main.liouville}, and $C^s$ norms are replaced with $Gevrey$ norms. 
\item[-] The conditions  $(B_\eta)(n) : |\bar{f}_{j,n}|< \eta u_n^n $  are replaced by  $|\bar{f}_{j,n}| < \eta u_n^{u_n^{-n}}.$ 
\item[-] The Liouville condition on the vectors  $(\bar F_1, \bar F_2)=(\bar{f}_{1,n}+\omega_1,\bar{f}_{2,n}+\omega_2)$  (as well as on $(\bar{f}_{1,n}+\omega_1,\bar{f}_{3,n}+\omega_3)$ and  $(\bar{f}_{2,n}+\omega_2,\bar{f}_{3,n}+\omega_3)$) is replaced by a "super-Liouville" condition of the type $|q_1 \bar F_1+q_2 \bar F_2|\leq e^{-q_1-q_2}$ for infinitely many $(q_1,q_2) \in \Z^2$. 

\end{itemize}

\section{Proof of the KAM counter term theorem}\label{sProof}
The proof of the counter term theorem (Proposition \ref{counterthm}) 
is based on an inductive procedure and will occupy this whole section. 

Let $\PP=\PP_{\k,\t}$ be the cut-off operator defined in section
\ref{s32}. We take $\t>d-1$ and $0<\k<1$. The operator $\PP$ depends on a cut-off function $l$ and 
constants in this section will, 
in general without saying, depend on $l$.
Recall that a function $g$ 
is $(\k,\t)${\it -flat} if
$$\p^\al_\f \p^\beta_z \p^\ga_\o g(\f,z,\o)=0$$
for all multi-indices $\a,\beta,\ga$ whenever $\o\in DC(\k,\t)$.

Let $B$ be a ball  centered at $\o_0$ or, more generally, the intersection of this
unit ball with an affine subspace of  $\R^d$ through $\o_0$.

\subsection{A linear operator}\label{s51}
Define now
$$\LL(f)=u$$
through
\begin{equation}
\left\lbrace\begin{array}{l}
\<\omega,\partial_{\f}u\>=f-\PP(f)-\MM(f)\\
\MM(u)=\PP(u)=0.
\end{array}\right.
\label{lineq}\end{equation}

\begin{Lem}\label{mainlemma}
$$\|\LL(f)\|_{\rho',\delta,s}\leq  
C_{s}(\frac1\kappa)^{s+1}(\frac{1}{\rho-\rho'})^{(\tau+1)(s+1)}
\|f\|_{\rho,\delta,s}$$
for any $\rho'<\rho$. The constant $C_s$ only depends, besides $s$, on $\t$ and $l$.
\end{Lem}

\begin{proof}
We give a proof with the exponent $(\tau+1)s+\t+d$ --  the improved exponent
$(\tau+1)s+\t+1$ requires some more subtle considerations 
originally due to R\"ussmann --  see for example \cite{E}.
Equation (\ref{lineq}) is equivalent to 
$\hat u(0,z,\o)=0$ and, for $n\in\Z^d-\{0\}$,
$$ \hat u(n,z,\omega)=\hat f(n,z,\omega) l_n(\o)$$
where
$$l_n(\o)=\frac1{i2\pi\l{n,\o}}(1- l(\lan n,\w \ran \frac{|n|^{\tau}}{\kappa} )).$$

Since
$$\aa{\hat f(n,\cdot,\cdot)}_{0,\de,s}\le \aa{f}_{\rho,\de,s} e^{-2\pi|n|\rho}.$$
and
$$\aa{l_n}_{0,0,s}\le C_s|n|^{(\t+1)s+\t}\frac1{\k^{s+1}}\aa{l}_{0,0,0}+  |n|^{\t}\frac1{\k}\aa{l}_{0,0,s},$$
we get (by Proposition \ref{hadamard}), for $|\al|\leq s$ and $(\f,z,\omega)\in  \T^d_{\rho'}\times \bD^{d}_{\delta}\times B$,
$$|\pa^\al_\o u(\f,z,\omega)|\leq C_{s} 
\sum_{n\ne 0}e^{2\pi |n|\rho'}\times$$
$$\times \biggl(\|\hat f(n,\cdot,\cdot)\|_{0,\de,s}
|n|^{\t}\frac1{\k}
+\|\hat f(n,\cdot,\cdot)\|_{0,\de, 0}|n|^{(\t+1)s+\t}\frac1{\k^{s+1}}\biggr)$$
which gives the estimates by standard arguments.
\end{proof}

\subsection{The counter term theorem}\label{s52}

Let $0<\rho,\de<1$.
Denote by
$\CC_{\rho,\delta}^{\w,\infty}$
the set of functions $f\in C^{\w,\infty}(\T^d_\rho \times \bD_{\delta}^d \times \bD_{\delta}^d,B)$
such that
$$f(\f,r,c,\o)\in \OO^2(r,c).$$

Any function $f\in  \CC_{\rho,\delta}^{\w,\infty}$ can be written 
uniquely as\footnote{\ we applogize for the double use of $B$}

$$a(\varphi,c,\w) + \langle  B(\f,c,\w),r-c \rangle+\frac12\langle r-c, F(\f,r,c,\w)(r-c) \rangle
$$
modulo $\OO^3(r-c)$ with $a=\cO^2(c)$ and $B=\cO(c)$. 

{ We say that $f$ is {\it of order $q$} if $a\in\cO^q(c)$ and  $B\in\cO^q(c)$.}

We define the pseudo-norm 
$$[f]_{\rho,\delta,s}=\max(\aa{a}_{\rho,\delta,s},\aa{B}_{\rho,\delta,s},
\aa{\p_\f\cL a}_{\rho,\delta,s},\aa{\p_\f\cL B}_{\rho,\delta,s})$$
and the vector
$$M_f=\MM  \left(B-F\partial_\f \LL a\right),$$
where, we recall, that
$\cM(g)$ is the {\it mean value} $\int_{\bT^d}g(\f,z)d\f.$
We denote by ${\EE}^{\w,\infty}_{\rho,\delta}$ the set of  
exact symplectic local diffeomorphisms     
defined on a neighborhood of $\T^d \times \{0\}$ of the form
 $$Z_{c,\o}(\f,r)=
\left(\begin{array}{l}
\f+\Phi(\f,c,\w) \\
r+R_1(\f,c,\w)+R_2(\f,c,\w) (r-c)
\end{array}\right)$$
with $\Phi, R_1,R_2  \in  \cC^{\w,\infty}(\T^d_\rho \times \bD_{\delta}^d \times \bD_{\delta}^d,B)$ and  $R_1=\cO^2(c)$,
$\Phi,R_2=\cO(c)$. If $Z'$ is another mapping in
${\EE}^{\w,\infty}_{\rho,\delta}$  then we define
$${[Z-Z']}_{\rho,\de,s}=$$
$$\max_i(\aa{\Phi-\Phi'}_{\rho,\de,s},\aa{R_i-R_i'}_{\rho,\de,s},
\aa{\p_\f\cL(\Phi-\Phi')}_{\rho,\de,s},\aa{\p_\f\cL(R_i-R_i')}_{\rho,\de,s})$$
and
$$(Z\circ Z')_{c,\o}(\f,r)=Z_{c,\o}( Z'_{c,\o}(\f,r)).$$

\medskip

The goal of this section is to prove the following 

\begin{Prop} \label{theo.kam}
For all $s\in\N$, there exist  constants 
$\eps>0$ and $\al(s)\ge 0$,  only depending on $\tau$, 
such that if  $H \in \CC_{\rho,\delta}^{\w,\infty}$ is independent of $\omega$
and satisfies,
for some  $h<\min(\rho/2,\delta/2)$ and some $\sigma<\eps(\tau)$,  
\begin{equation} \label{e510}   
[H]_{\rho,\de,0}\le
\sigma\frac1{(1+\aa{\p_r^2H}_{\rho,\de,0})^7}\k^{11}h^{10(\tau+d)+11},\end{equation} 

then there exist $\Lambda \in \CC^{\w,\i}_{0,\de-h}$ and  $W\in \EE^{\w,\i}_{\rho-h,\de-h}$,    
$H' \in \cC^{\w,\infty}_{\rho-h,\de-h}$, with $[H']_{\rho-h,\de-h,0}=0$, 
and  a $(\k,\t)$-flat function  
$g \in \cC^{\w,\infty}_{\rho-h,\de-h}$  such that 

\begin{equation} \label{ered} (H+\langle \o+ \Lambda(c,\o), \cdot \rangle) \circ W_{c,\o} (\varphi,r) = 
\langle \w, r-c \rangle +  H'(r,\f,c,\w) +  g(\f,r,c,\w) \end{equation} 
(modulo an additive constant that depends on $c,\omega$) 
with, for all $s$, 

\begin{multline}
\max \left(\aa{\Lambda}_{0,\delta-h,s}, {[W-\id]}_{\rho-h,\de-h,s}, \aa{g}_{\rho-h,\de-h,s},  
{\aa{\p_r^2(H'-H)}}_{\rho-h,\de-h,s},    \right) \\ \label{etailles} 
< \sigma
\left(\frac{\aa{\p_r^2H}_{\rho,\de,0}+1}{\k h}\right)^{\alpha(s)}
(\aa{\p_r^2H}_{\rho,\de,0}+[H]_{\rho,\de,0}+1).
\end{multline}

{ Moreover, if $H$ is of order $q$, then $g\in\cO^q(c)$.}

Furthermore, if  
$$\o_0\in DC(2{\k},\t)$$
then $\Lambda, W$ and  $H' $  are analytic on $I_{\de'}$ for some $0<\de'\leq \de$ and 
$g=0$ on $I_{\de'}$.
\end{Prop}

We shall first prove the main part of this proposition, then we will explain what
modifications are required in order to obtain the final analyticity statement.

The proof of Proposition \ref{theo.kam} is based on an inductive KAM scheme. 
In each step of the scheme we  conjugate a Hamiltonian of the form 
$$\langle \o,r-c \rangle+a(\varphi,c,\w) + \langle  B(\f,c,\w),r-c \rangle+
\frac12\langle r-c,
F(\varphi,r,c,\w)(r-c)\rangle.$$
and reduce quadratically the terms $a$ and $B$. To do so, we look for a conjugacy 
using a generating function of the form $\<r,\psi\>+u_0(\psi)+\<u_1(\psi),r\>$ and 
we solve a triangular cohomological system  in $u_0$ and $u_1$ to reduce $a$ and $B$. 
This is only possible up to a $(\kappa,\tau)$-flat  function $g$ and also requires 
that the constant terms in the cohomological equations vanish and this is why we 
have to add the counter term  
$\<\Lambda,\cdot \>$ and a constant.

The inductive step of the scheme is enclosed in Proposition \ref{inductive}. 
To add clearness to the presentation we split the proof of the latter proposition 
into two parts : in the first part we suppose the constant terms in the 
cohomological equations do vanish and build the conjugacy (this is the content of 
Lemma \ref{sanslambda}) and in the second one we show that adding counter  terms 
allows to zero the constant terms in the cohomological  equations (this is the content 
of Lemma \ref{lambda}).  Proposition \ref{inductive} is a direct consequence of Lemmas 
\ref{sanslambda} and \ref{lambda}.

We will finally conclude in Sections \ref{s55} and \ref{s56} showing that the iteration scheme based on the inductive step 
of Proposition \ref{inductive} does converge if the initial bound (\ref{e510})    is satisfied.

\subsection{Reduction lemmas}\label{s53}

In this section we first fix $\rho,\de<1$ and 
a number $h$ less than $\min(\rho/2,\delta/2)<\frac12$
and we set
$$\xi_s=\k^{(s+1)}  h^{(\t+1)(s+1)+d}.$$
We fix $H\in \CC_{\rho,\delta}^{\w,\infty}$, which may depend on $\omega$, and let
$$\eps_{s} ={[H]}_{\rho,\delta,s}
\quad\textrm{and}\quad
\zeta_{s} = \aa{\p_r^2 H}_{\rho,\de,s}+1.$$

\medskip

\begin{Lem} \label{sanslambda} 
There exist positive constants  $\varsigma=\varsigma(\tau)$ 
and $C_s=C_s(\tau)$, such that if $M_H = 0$
and 
\begin{equation} \label{e511}  \eps_{1}< \varsigma 
\frac1{\zeta_1} \kappa^2 h^{2\t +d+6},\end{equation} 
then there exist
$Z\in \EE^{\w,\i}_{\rho-h,\de-h}$, $H' \in \cC^{\w,\i}_{\rho-h,\de-h}$ and a $(\k,\t)$-flat function 
$g'$ such that 
\begin{equation*}(H+ \langle \w, \cdot \rangle)  \circ Z_{c,\o} (\varphi,r) = 
\langle \w, r-c\rangle + H'(\varphi,c,\w)+ g'(\f,r,c,\w),\end{equation*}
(modulo an additive constant that depends on $c,\o$) with, for all $s\in\N$,
$$[Z-\id]_{\rho-h,\delta-h,0} < \frac{h}{2},$$
$$ {[H']}_{\rho-h,\delta-h,s}\le \nu_s\eps_0$$
and
$$\max({[Z-\id]}_{\rho-h,\delta-h,s},\aa{\p_r^2(H'-H)}_{\rho-h,\delta-h,s},\aa{g'}_{\rho-h,\delta-h,s})
\le\nu_s,$$
for
$$\nu_{s}=C_s \frac{1}{\xi_s^{3}}
\zeta_{0}  (\zeta_0 \eps_{s}+ \zeta_{s}  \eps_{0}).$$
{Moreover, if $H$ is of order $q$, then $H'$ is of order $q$ and $g'\in\cO^q(c)$.}

\end{Lem}

\begin{proof} 
We  introduce $G(\f,r,c,\o)\in\cO^3(r-c)$ 
$$G:=H(\f,r,c,\o) - a(\f,c,\o) -\langle B(\f,c,\o),r-c\rangle-
\frac12\langle r-c,F(\varphi,c,\w)(r-c)\rangle.$$
Notice that 
$$
\aa{\p_r^2 G}_{\rho,\de,s}\lsim  \aa{\p_r^2 H}_{\rho,\de,s}$$
and
$$
\aa{\p_r^3 G}_{\rho,\de-h,s}\lsim  \frac{1}{h} \aa{\p_r^2 H}_{\rho,\de,s}.$$

We look for the diffeomorphism $Z(\f,r,c,\o)=(\psi,s)$ {\it via} a 
generating function of the form $\langle \psi,r \rangle + U(\psi,c,\w)$,
$$U(\psi,r,c,\w)=u_0(\psi,c,\w)+  \langle u_1(\psi,c,\w),r-c \rangle 
 \in \cC^{\w,\infty}_{\rho-h,\delta-h,s},$$
i.e. 
$$\left\{\begin{array}{l}
s=r+ \partial_\psi u_0  + \l{\p_\psi u_1, r-c} \\
\f=\psi+u_1(\psi). \end{array}\right.$$
All our functions depend, besides $\psi$, on $c,\o$ and we shall in the sequel suppress this
dependence  in the notations.

We have, modulo an additive constant,
\begin{multline}\label{e512}
(H+ \langle \w, \cdot \rangle)\circ Z(\f,r)-\< \w,r-c\> \\
= (I) + (II) + (III) +
 G(\psi,r+\partial_\psi U(\psi,r)),
\end{multline}
where
 \begin{align*}
 (I) &=  \l{\w,\p_\psi u_0} +a+\l{B,\p_\psi u_0}+ \frac12\l{F \partial_\psi u_0,\partial_\psi u_0}-\cM(a) \\ 
(II) &=  \l{ \l{\omega,\p_\psi u_1}, r-c} +  \<B,r-c+\p_\psi u_1 (r-c)\>  \\
&\ \  \ +  \l{F \partial_\psi u_0,r-c+\partial_\psi u_1 (r-c)}     \\
(III)&= \frac12\langle r-c,F(r-c)\rangle +  \l{F (r-c),\partial_\psi u_1 (r-c)} \\
&\ \ \ +  \frac12\l{F\partial_\psi u_1 (r-c), \partial_\psi u_1 (r-c)}.   
\end{align*}

{\it The homological equation.}
To kill as much as possible of $a$ and $B$ in $(H+ \langle \w, \cdot \rangle)\circ Z$ we take
 $$
\left\{\begin{array}{l}
u_0=- \LL(a) \\
u_1= -\LL(B+F  \partial_\psi u_0)     \end{array}\right.$$
{ Observe  that if $a,B\in\cO^q(c)$, then $u_0, u_1\in\cO^q(c)$.} Recall also that the control on $[H]_{\rho,\delta,s}$ implies a control on $\aa{\p_\f\cL a}_{\rho,\delta,s},\aa{\p_\f\cL B}_{\rho,\delta,s}$. Define 
$$e_s:= \frac{1}{h\xi_s} (\zeta_s \eps_0+\zeta_0\eps_s).$$

By Lemma \ref{mainlemma} we have\footnote{\  The constant $C_s$ will differ from line to line}
\begin{equation*}\label{equ1}
[U]_{\rho-h,\de,s}+[\p_r U]_{\rho-h,\de,s}\leq C_s e_s.
\end{equation*}
(Here and elsewhere we use Proposition \ref{hadamard} to estimate products. The additional factro $1/\xi_s$ appears since we have to apply the operator $\cL$ to $\delta_\psi \cL a$.)

{\it Estimation of $Z$.} Since, by (\ref{e511}),
$$e_1\ll h$$
Proposition \ref{annexe.inverse} implies  that the mapping $\f=\tilde f(\psi)=\psi+u_1(\psi)$
is invertible with inverse satisfying
\begin{align*}   
\aa{\tilde f^{-1}-\id}_{\rho-2h,\de -h,s} \le C_s e_s,
\end{align*}
and Proposition \ref{annexe.compose} gives
\begin{align*}   
\aa{Z-\id}_{\rho-3h,\de -2h,s} \leq   C_s  \frac{e_s}{h}.
\end{align*}
It follows that $Z  \in \EE_{\rho-4h,\de-3h}^{\w,\i}$ and Lemma \ref{mainlemma} implies that
\begin{align*}   
{[Z-\id]}_{\rho-4h,\de -3h,s} \leq   C_s \frac{e_s}{ \xi_s h^{2}}.
\end{align*}

\medskip

{\it Estimation of the function $g'$.}
Let 
$$h=\PP(a) +  \langle \PP(B-F  \partial_\psi u_0),r-c \rangle$$
and $g'(\f,r)=h(\psi,r)$. Then, by Lemma \ref{flat},
$$\aa{h}_{\rho-2h,\de,s} \le C_s e_s$$
and, by Proposition \ref{annexe.compose},
$$\aa{g'}_{\rho-3h,\de-2h,s} \le 
C_s e_s h^{-1}.$$

\medskip 

{\it Checking that $H'$ is of order $q$.} We note that if $H$ is of order $q$, then $a,B,u_0,u_1,g'\in \cO^q(c)$.
Hence the terms (I) and (II) in the RHS of 
\eqref{e512} are $\cO^q(c)$. Now
$H'(\f,r)$ equals
$$I(\f+\Phi(\f,c),c)+II(\f+\Phi(\f,c),r-c,c)+III(\f+\Phi(\f,c),r-c,c)$$
and since $III\in \cO^2(r-c)$ we conclude that $H'$ is of order $q$.

\medskip 

{\it Estimation of $H'$.} We set 
$$G_1(\f,r)=\langle \p_r G(\f,r), \p_\psi U(\f,r)\rangle$$
and 
$$G_2(\f,r)=G(\psi,r+\partial_\psi U(\psi,r))-G(\f,r)-\langle \p_r G(\f,r),\p_\psi U(\f,r)\rangle.$$
Then  $G_1\in \cO^2(r-c)$  and   the RHS of \eqref{e512} satsifies 
$$RHS-h=(I)+(II)-h+G_2+\cO^2(r-c)$$ 
as well as 
$$\p_r^2(RHS-H)=\p_r^2((III)+G_1+G_2)-F$$
because $G=\cO^3(r-c)$. 
Now we have that 
\begin{equation}\label{e513}
[(I)+(II)-h]_{\rho-2h,\de,s} \leq C_s\frac1{h^{2}\xi_s}[(\eps_s e_0  + \eps_0 e_s )
+e_0(\zeta_s e_0  + \zeta_0 e_s )]
\end{equation}
(here we use that $M_H=0$) and by Proposition \ref{annexe.compose}(ii),
\begin{equation}\label{e516}
\aa{G_2}_{\rho-2h,\de-h,s} \leq C_s(\zeta_s e_0 + \zeta_0 e_s)\frac{e_0}{h^{4}}\end{equation}
since $e_0\ll h^{4}$.
 The inequality (\ref{e516}) implies in particular that
\begin{equation}\label{e517}
[G_2]_{\rho-3h,\de-2h,s} \leq C_s\frac1{h^{2}\xi_s}(\zeta_s e_0 + \zeta_0 e_s)\frac{e_0}{h^{4}}\end{equation}
and
\begin{equation}\label{e518}
\aa{\p_r^2 G_2}_{\rho-2h,\de-2h,s} \leq C_s(\zeta_s e_0 + \zeta_0 e_s)\frac{e_0}{h^{6}}.\end{equation}

It follows from (\ref{e513}) and (\ref{e517})  that the right hand side of (\ref{e512}) verifies
$$
[RHS-h]_{\rho-3h,\de-2h,s} \leq 
C_s\frac1{h\xi_s}[(\eps_s e_0  + \eps_0 e_s )\frac{1}{h}
+C_s(\zeta_s e_0  + \zeta_0 e_s )\frac{e_0}{h^5}].$$

On the other hand, since  \begin{equation}\label{e515}
\aa{\p_r^2 G_1}_{\rho-2h,\de-h,s} \leq C_s(\zeta_s e_0 + \zeta_0 e_s)\frac{1}{h^2}
\end{equation}
and
\begin{equation}\label{e514}
\aa{\p_r^2(III)-F}_{\rho-2h,\de,s} \leq C_s(\zeta_s e_0 + \zeta_0 e_s)\frac{1}{h^{2}}.
\end{equation}
it follows that
$$
\aa{\p_r^2(RHS-H)}_{\rho-2h,\de-2h,s} \leq C_s(\zeta_s e_0 + \zeta_0 e_s)\frac{e_0}{h^{6}}$$
--  by (\ref{e518}+\ref{e515}+\ref{e514}).

Since $H'(\f,r)=RHS(\psi,r)-g'(\f,r)$ we get by Proposition \ref{annexe.compose} that
$$
[H']_{\rho-4h,\de-3h,s} \leq 
C_s\frac1{h\xi_s}[(\eps_s e_0  + \eps_0 e_s )\frac{1}{h}
+C_s(\zeta_s e_0  + \zeta_0 e_s )\frac{e_0}{h^{5}}]$$
and
$$
\aa{\p_r^2(H'-H)}_{\rho-4h,\de-3h,s} \leq C_s(\zeta_s e_0 + \zeta_0 e_s)\frac{e_0}{h^{6}}.$$

This completes the proof of the lemma.
\end{proof} 

Let $W \in \EE_{\rho,\delta}^{\w,\infty}$ and denote
$$\eta_s=[W-\id]_{\rho,\delta,s}.$$

\begin{Lem} \label{lambda} There exist positive constants  $\varsigma=\varsigma(\tau)$ 
and $C_s=C_s(\tau)$, such that  if 
\begin{equation} \label{e519} 
\eta_0<  \varsigma \frac1{\zeta_0},\end{equation}
then there exists $\Lambda \in \CC^{\w,\i}_{\de}$ such that
$$ \tilde H=H+\langle \Lambda, \cdot \rangle \circ W$$
verifies $ M_{\tilde H} = 0$ and $\p_r^2 \tilde H=\p_r^2 H$.

{ Also, if $H$ is of order $q$, then $\tilde{H}$ is of order $q$.}

Moreover, for all $s\in\N$,
\begin{equation*} 
\aa{\Lambda}_{\delta,s}\le C_s\zeta_0
(\zeta_0\eps_0\eta_s+\zeta_0\eps_s+\zeta_s\eps_0) \end{equation*}
and
\begin{equation*} 
[\tilde H-H]_{\rho,\delta,s} \le C_s\zeta_0
(\zeta_0\eps_0\eta_s+(\zeta_0\eps_s+\zeta_s\eps_0)(\eta_0+1)).
\end{equation*}
\end{Lem}

\begin{proof} 
Write $H(\varphi,r,c,\o)$ as
$$a(\varphi,c,\w) + \langle  B(\f,c,\w),r-c \rangle+\frac12\langle r-c, F(\f,c,\w)(r-c) \rangle$$
and
$\tilde H(\varphi,r,c,\o)$ as
$$\tilde a(\varphi,c,\w) + \langle  \tilde B(\f,c,\w),r-c \rangle+\frac12\langle r-c, \tilde F(\f,c,\w)(r-c) \rangle $$
modulo $\OO^3(r-c)$.

Observe that 
 $$W_{c,\o}(\f,r)=
\left(\begin{array}{l}
\f+\Phi(\f,c,\w) \\
r+R_1(\f,c,\w)+R_2(\f,c,\w) (r-c)
\end{array}\right),$$
so
\begin{align} \label{ea1}
\tilde a&=a+\<\Lambda,R_1+c\> \\
\tilde  B&=B+(I+{}^t\! R_2) \Lambda \label{eb1}
\end{align}
and $\tilde F=F$. We want to choose $\Lambda$ so that $M_{\tilde H}=0$, i.e.
 $$\MM\left( \left[I+{}^t\! R_2-F\p_\f \LL R_1\right] \Lambda - F \partial_\f \LL a +B\right)=0.$$
 If $X=-\MM({}^t\! R_2-F\p_\f \LL R_1)$ and $Y=\MM(-B + F \partial_\f \LL a)$, then
 this amounts to
 \begin{equation} \Lambda=  \sum_nX^nY. \label{el1} \end{equation}

{  Observe that if $a,B\in\cO^q(c)$, then $Y\in \cO^q(c)$, thus 
$\Lambda\in\cO^q(c)$ and $\tilde{H}$ is of order $q$. }

We have
$$\aa{X}_{\de,s}\le C_s(\zeta_s\eta_0+\zeta_0\eta_s)$$
 and
$$\aa{Y}_{\de,s}\le C_s(\zeta_s\eps_0+\zeta_0\eps_s).$$

By assumption (\ref{e519}), $\aa{X}_{\de,0}\le 1/2$, which gives the existence
and the estimates on $\Lambda$ and $\tilde{H}$ by Proposition \ref{hadamard} and (\ref{ea1})--(\ref{el1}).
\end{proof}

\subsection{The inductive step}\label{s54}

Combining Lemma \ref{sanslambda} and  Lemma \ref{lambda} 
we immediately get the following proposition that constitutes 
the inductive step of our KAM scheme for the proof of Proposition \ref{theo.kam}.
For the needs of the inductive application, we will consider  that 
at each step we have a Hamiltonian $H \in  \cC^{\w,\i}_{\rho,\de}$ as well as $g \in   \cC^{\w,\i}_{\rho,\de}$ $(\k,\t)$-flat, and $W\in   \EE^{\w,\i}_{\rho,\de}$.

As in the previous section we  assume
$h<\min(\rho/2,\delta/2)$, and we set
$$\xi_s=\k^{(s+1)}  h^{(\t+1)(s+1)+d}$$
and 
$$\eps_{s} ={[H]}_{\rho,\delta,s}
\quad\textrm{and}\quad
\zeta_{s} = \aa{\p_r^2 H}_{\rho,\de,s}+\aa{g}_{\rho,\de,s}+1$$
and
$$\eta_s=[W-id]_{\rho,\de,s}.$$

\begin{Prop} \label{inductive}  There exist $\varsigma=\varsigma(\tau)$ 
and $C_s=C_s(\tau)$ such that, if 

\begin{equation} \label{e520} 
\eta_{0}  <  \varsigma \frac1{\zeta_0}\end{equation}
and
\begin{equation} \label{e521}  \eps_{1}< 
\varsigma \frac1{\zeta_1^2(1+\eta_1)}\kappa^2  h^{2\t +d+6}, \end{equation} 
then there exist $\Lambda \in \CC^{\w,\i}_{0,\de}$, $Z'\in \EE^{\w,\i}_{\rho-h,\de-h}$, 
$H' \in \cC^{\w,\i}_{\rho-h,\de-h}$ and a $(\k,\t)$-flat function $g'$ such that 
 \begin{multline}\label{e522}
(H+ g+\langle \o, \cdot \rangle +\langle \Lambda(c,\o), 
\cdot \rangle \circ W)\circ Z_{c,\o}'(\varphi,r) = \\
\langle \w, r-c \rangle +H'(\varphi,c,\w) +  g'(\f,r,c,\w)\end{multline} 
(modulo an additive constant that depends on $c,\o$) with
\begin{equation}\label{e523}
[Z'-\id]_{\rho-h,\delta-h,0} < \frac{h}{2},
\end{equation}
\begin{equation}\label{e524}
{[H']}_{\rho-h,\delta-h,s}\le \nu_s\eps_0
\end{equation}
and
\begin{multline}\label{e525}
\max(\aa{\Lambda}_{0,\delta-h,s},\aa{\p_r^2(H'-H)}_{\rho-h,\delta-h,s},\aa{g'-g}_{\rho-h,\delta-h,s},\\
{[Z'-\id]}_{\rho-h,\delta-h,s},[W\circ Z'-W]_{\rho-h,\delta-h,s})\le\nu_s,
\end{multline}
where 
\begin{equation}\label{e526}
\nu_{s}=C_s \xi_s^{-5}
\zeta_{0}^{5}  (\eps_s+\zeta_s\eps_{0}+ \eta_{s} \eps_{0}).
\end{equation}

{Moreover, if $H$ is of order $q$ and $g\in\cO^q(c)$, then $H'$ is of order $q$ 
and $g'\in\cO^q(c)$.}

\end{Prop}

\begin{Rem}
Notice that the assumption (\ref{e520}) follows if
$$\eta_{1}  <  \varsigma \frac1{\zeta_1},$$
and that
$$\nu_s\eps_0\le \nu_s\eps_1$$
and
$$
\nu_{s}\le C_s \xi_s^{-5}
\zeta_{1}^{5}  (\eps_s+\zeta_s\eps_{1}+ \eta_{s} \eps_{1}).$$
\end{Rem}

\begin{proof}
By Lemma \ref{lambda} there exists $\Lambda \in \CC^{\w,\i}_{\de}$,
\begin{equation*}
\aa{\Lambda}_{0,\delta,s}\le 
C_s\zeta_0(\zeta_0\eps_0\eta_s+\zeta_0\eps_s+\zeta_s\eps_0),
\end{equation*}
such that
$$ \tilde H=H+\langle \Lambda, \cdot \rangle \circ W$$
verifies $ M_{\tilde H} = 0$, 
\begin{equation*} 
[\tilde H]_{\rho,\delta,s}\le  \eps_s+
C_s\zeta_0(\zeta_0\eps_0\eta_s+\zeta_0\eps_s+\zeta_s\eps_0)= \tilde\eps_s
\end{equation*}
and
\begin{equation*} 
\tilde\zeta_s=\aa{\p_r^2 \tilde H}_{\rho,\delta,s}+\aa{g}_{\rho,\de,s}+1=\zeta_s.
\end{equation*}

Since by \eqref{e521} 
$$
\tilde\eps_1\le \varsigma \frac1{\tilde\zeta_1} \kappa^2 h^{2\t +d+6},$$
Lemma \ref{sanslambda} gives
$Z'\in \EE^{\w,\i}_{\rho-h,\de-h}$, $H' \in \cC^{\w,\i}_{\rho-h,\de-h}$ and a $(\k,\t)$-flat function 
$g''$ such that 
\begin{equation*}(\tilde H+ \langle \w, \cdot \rangle)  \circ Z' (\varphi,r,c,\o) = 
\langle \w, r-c\rangle + H'(\varphi,c,\w)+ g''(\f,r,c,\w),\end{equation*}
(modulo an additive constant that depends on $c,\o$) with, if we let 
$$\tilde \nu_{s}:=C_s \xi_s^{-3}\zeta_{0}  (\zeta_0 \tilde \eps_{s}+ \zeta_{s}  \tilde \eps_{0}),$$
$$[Z'-\id]_{\rho-h,\delta-h,0} <    \frac{h}{2},$$
$$ {[H']}_{\rho-h,\delta-h,s}\le \tilde \nu_s\tilde \eps_0$$
and
$$\max({[Z'-\id]}_{\rho-h,\delta-h,s},\aa{\p_r^2(H'-\tilde H)}_{\rho-h,\delta-h,s},\aa{g''}_{\rho-h,\delta-h,s})
\le\tilde \nu_s.$$

Since $g'=g\circ Z'+g''$ we get that   $g'$ is flat and Proposition \ref{annexe.compose} implies that
$$\aa{g'-g}_{\rho-2h,\delta-2h,s} \leq C_s h^{-1} \zeta_s \tilde \nu_s$$

If we write $W=\id+f$ and $Z'=\id+f'$, then
$$W\circ Z'-W=f'+(f\circ (\id+f')-f).$$
We have already seen that ${[f']}_{\rho-h,\delta-h,s}\le\tilde\nu_s$
and  Lemma \ref{mainlemma} and Proposition \ref{annexe.compose} imply
$${[f\circ(\id+f')-f]}_{\rho-2h,\delta-2h,s}\le (h^{2} \xi_s)^{-1}\tilde\nu_s.$$ 
The conclusions of the lemma then follow with $\nu_s$ as in \eqref{e526}.
\end{proof}

\subsection{Convergence of the KAM scheme}\label{s55} 

We will show in Section \ref{s56} that the inductive application of Proposition \ref{inductive} yields Proposition \ref{theo.kam}. Before this, we show in the current section two computational lemmas that will allow, under condition \eqref{e510} of Proposition \ref{theo.kam},  to apply inductively Proposition \ref{inductive} by checking conditions \eqref{e520} and \eqref{e521} 
at each step, and get the required estimates of Proposition \ref{theo.kam}. The first lemma deals with $C^1$ norms relative to $\omega$, while the second one contains the estimates relative to the higher order norms. 
 
\begin{Lem} \label{quadratic}  Fix $0<h<\frac12$ and let $h_n=h2 ^{-n-1}$. Let
$a,b,c$ and $C\ge 0$ and let there 
be given four non negative sequences $\nu_n,\zeta_n,\eta_n,\eps_n$ such that 
\begin{equation}\label{e527}
\nu_n \le  C \k^{-2b} h_{n}^{-2a} \zeta_n^{c}(\zeta_n+\eta_n) \eps_n
\end{equation}
for $n\ge0$ and
\begin{align} 
\label{e528}\zeta_n &\leq \zeta_{n-1} + \nu_{n-1}&& \zeta_0=\zeta\ge1 \\
\label{e529}\eta_n &\leq \eta_{n-1} +  \nu_{n-1}&& \eta_0=0 \\
\label{e530}\eps_n &\leq   \nu_{n-1} \eps_{n-1}&& \eps_0=\eps
\end{align}
for $n\ge1$.

Then there exists $C'=C'(C,a,b,c)>0$ such that if for $\varsigma\leq 1$  
\begin{equation}\label{e531}
\eps< \frac{\varsigma}{C'} \kappa^{2b+1} h^{2a+1}\zeta^{-c-2}\end{equation}
then \begin{align}
\label{e533} &\eps_n\leq  (\kappa h \zeta^{-1})^{2^n-1} \eps \\
\label{e534} &\eta_n <\varsigma \zeta_n^{-1}.
\end{align}
\end{Lem} 

\begin{proof}
Assume $\zeta_n\le A=2\zeta$ and $\eta_n\le 1$ for all $n$. Then
$$\nu_n\le BD^{n}\eps_{n}$$
with $B=C \k^{-2b} h^{-2a} 2^{2a+1} A^{c+1}$ and $D=4^{a}$. Hence for $n\geq 1$ we have 
$$\eps_n\le BD^{n-1}\eps_{n-1}^2
\le B^{2^n-1}D^{2^{n}-n-1}\eps^{2^n}=
\frac1{BD^{n+1}}(BD\eps)^{2^n}$$
which shows (\ref{e533}) if $C'$ is sufficiently large.

From \eqref{e533} we get that 
$$\sum_n\nu_n\le  \frac{\varsigma}{2\zeta}, $$
and the assumptions $\zeta_n\le A=2\zeta$ and $\eta_n\le 1$ --- actually (\ref{e534}), now follow by induction.
\end{proof}

\begin{Lem} \label{convergences}
Fix $0<h<\frac12$ and let $h_n=h 2 ^{-n-1}$. Let
$a,b,c\ge 0$ and suppose $\eps_n$ is a sequence satisfying (\ref{e533}) with $\eps=\eps_0$ verifying \eqref{e531} of Lemma \ref{quadratic}.
Assume that $C_s\ge 0$, and that four sequences 
$\nu_{s,n},\zeta_{s,n},\eta_{s,n},\eps_{s,n}$ satisfy 
\begin{equation}\label{e536}
\nu_{s,n} \le C_s \k^{-b(s+1)} h_{n}^{-a(s+1)}\zeta^c ( \eps_{s,n}+\zeta_{s,n} \eps_{n}
+\eta_{s,n} \eps_{n})
\end{equation}
for all $n\ge0$, and
\begin{align} 
\label{e537}\zeta_{s,n} &\leq \zeta_{s,n-1} + \nu_{s,n-1}&& \zeta_{s,0}=\zeta \geq 1 \\
\label{e538}\eta_{s,n} &\leq \eta_{s,n-1} + \nu_{s,n-1}&& \eta_{s,0}=0\\
\label{e539}\eps_{s,n} &\leq  \nu_{s,n-1}\eps_{n-1}&& \eps_{s,0}=\eps
\end{align}
for all $n\ge1$. Then 
\begin{equation} \label{e542}
\sum_{n\ge0}\nu_{s,n}\le \sigma(\k^{-1} h^{-1}\zeta)^{\alpha(s)}(\zeta+\eps)
\end{equation}
where $\sigma := \frac{\varsigma}{C'}$ of \eqref{e531} and   $\alpha(s)$ is some increasing function in $s$ depending on $C_s$ and $a,b,c$.
\end{Lem}

\begin{proof} By replacing $\zeta_{s,n}$ by $\zeta_{s,n}+\eta_{s,n}$
we see that it is enough to consider the case $\eta_{s,n}=0$ for all $n$. 

If we let $U_s=\bar{C} \kappa^{-b(s+1)} h^{-a(s+1)}\zeta^{c}$,
with $\bar{C}(s,a,b,c,C)>0$ sufficiently large,  then it is immediate by induction that
\begin{equation} \label{nuu} \max(\eps_{s,n},\zeta_{s,n}-\zeta,\nu_{s,n}) \leq \sigma U_s^{n+1} (\zeta+\eps)
\end{equation}
Thus, if $n \geq N(s) \gg \max (\log(s+1),\log C_s)$, (\ref{e539})  and  (\ref{nuu}) and \eqref{e533} imply that 
\begin{align*}  C_s \k^{-b(s+1)} h_{n}^{-a(s+1)}\zeta^c \eps_{s,n}&\leq \sigma U_s^{2n+1}  (\kappa h \zeta^{-1})^{2^{n-1}-1} \eps   (\zeta+\eps) \leq  \frac{\sigma}{2^n}  (\zeta+\eps) \\ 
 C_s \k^{-b(s+1)} h_{n}^{-a(s+1)}\zeta^c \zeta_{s,n} \eps_n&\leq \sigma U_s^{2n+1}  (\kappa h \zeta^{-1})^{2^n-1} \eps  (\zeta+\eps) \leq  \frac{\sigma}{2^n}  (\zeta+\eps)  \end{align*}
hence, for $n\geq N(s)$ we get that 
\begin{equation} \label{nuu2} 
\nu_{s,n} \leq  \frac{\sigma}{2^{n-1}}  (\zeta+\eps)
\end{equation} and (\ref{e542}) follows, with $\a(s)= (s+1) (N(s))^2$, if we sum  (\ref{nuu}) for $n$ from $0$ to $N(s)$ and (\ref{nuu2}) for $n\geq N(s)$.
\end{proof}

\subsection{Proof of Proposition \ref{theo.kam}}\label{s56}   

We now prove Proposition  \ref{theo.kam} from an inductive application of Proposition \ref{inductive}. 

Let $h<\min(\rho,\delta)/2$, $h_n=h 2 ^{-n-1}$ and define
$$\rho_n= \rho-\sum_{i< n} h_i,\quad \de_{n}= \de-\sum_{i< n} h_i$$
and
$$\xi_{s,n}=\k^{(s+1)}  h_n^{(\t+1)(s+1)+d}.$$

We start by setting $H_0=H$, $\Lambda_0=0$, $g_0=0$  and $W_0=\id$, and we shall define inductively $H_n$, $\Lambda_n$, $Z_n$ and $W_n=Z_0\circ Z_1\circ \dots \circ Z_n$. In light of (\ref{e526}), let
$$\eps_{s,n} ={[H_n]}_{\rho_n,\delta_n,s},\quad
\zeta_{s,n} = \aa{\p_r^2 H_n}_{\rho_n,\de_n,s}+ \aa{g_n}_{\rho_n,\de_n,s}+1$$
$$
\eta_{s,n} ={[W_n-\id]}_{\rho_n,\delta_n,s}, \quad \nu_{s,n}=C_s \xi_{s,n}^{-5}
\zeta_{0,n}^{5}  (\eps_{s,n}+\zeta_{s,n}\eps_{0,n}+ \eta_{s,n} \eps_{0,n})
$$
where $C_s$ is given by Proposition \ref{inductive}. 
We fix hereafter $a=5(\tau+1+d), b=5,c=5$ and we will apply  Lemmas \ref{quadratic} and  \ref{convergences} with these values and with $C_s$ as in Proposition \ref{inductive}, while $C$ of Lemma \ref{quadratic} is just $C_1$. As for $\varsigma$, we will take it as $\varsigma(\tau)$ of Proposition \ref{inductive}.

Note also that for $s=0$, the fact that $H_0=H$ does not depend on $\omega$, hence $\zeta_{s,0}=\zeta_{0,0}=\zeta$, $\eps_{s,0}=\eps_{0,0}=\eps$ as required by Lemma \ref{convergences}.  By assumption also we have $\eta_{s,0}=0$.   To finish with the initial conditions, it follows from (\ref{e510}) that $\eps$ verifies conditions (\ref{e531}), provided $\eps(\tau)$ of Proposition \ref{theo.kam} is taken sufficiently small.

Based on (\ref{e524}--\ref{e526}), we assume by induction that, for $j=0,\dots, n$, $\nu_{s,j},\zeta_{s,j},\eta_{s,j},\eps_{s,j}$
verify (\ref{e527}--\ref{e530}) for $s=1$, and  (\ref{e536}--\ref{e539})  for 
$s\ge1$. 

Then, by  (\ref{e533}) and (\ref{e534}) of Lemma \ref{quadratic} we verify that, at each step $n$, conditions (\ref{e520}) and (\ref{e521}) of Proposition \ref{inductive} are satisfied 
so that we can apply the latter proposition and get $\Lambda_{n+1} \in \CC^{\w,\i}_{0,\de_{n+1}}$, $Z_{n+1}\in \EE^{\w,\i}_{\rho_{n+1},\de_{n+1}}$, 
$H_{n+1} \in \cC^{\w,\i}_{\rho_{n+1},\de_{n+1}}$ and  $g_{n+1} \in  \cC^{\w,\i}_{\rho_{n+1},\de_{n+1}}$ $(\k,\t)$-flat such that 
 \begin{multline*}
(H_n+ \langle \o, \cdot \rangle +\langle \Lambda_{n+1}, \cdot \rangle \circ W_n +g_n )\circ Z_{n+1}(\varphi,r,c,\o) = \\
\langle \w, r-c \rangle +H_{n+1}(\varphi,c,\w) +  g_{n+1}(\f,r,c,\w)\end{multline*} 
(modulo an additive constant). Moreover,  by (\ref{e524}--\ref{e526}) we have that 
\begin{multline}\label{e543}
\max(\aa{\Lambda_{n}}_{\rho_{n},\delta_{n},s},\aa{\p_r^2(H_{n}-H_{n-1})}_{\rho_{n},\delta_{n},s},\\
\aa{g_{n}-g_{n-1}}_{\rho_{n},\delta_{n},s},
{[Z_n-\id]}_{\rho_n,\delta_n,s},[W_{n-1}\circ Z_n-W_{n-1}]_{\rho_n,\delta_n,s}) \leq \nu_{s,n}\end{multline}
and that $\nu_{s,n+1},\zeta_{s,n+1},\eta_{s,n+1},\eps_{s,n+1}$
satisfy (\ref{e527}--\ref{e530}) for $s=1$, and  (\ref{e536}--\ref{e539})  for 
$s\ge1$.
Finally, \eqref{e542} gives that 

\begin{equation*} \label{e5422}
\sum_{n\ge0}\nu_{s,n}\le \sigma(\k^{-1} h^{-1}\zeta)^{\alpha(s)}(\zeta+\eps)
\end{equation*}
which together with (\ref{e543}) show that $\sum \Lambda_{l} = \Lambda \in \CC^{\w,\i}_{{\de}-h}$, $W_{n}$ converges to $ W \in \EE^{\w,\i}_{{\rho}-{h},{\delta}-{h}}$,  and $H_{n}$ converges to  $H' \in \cC^{\w,\i}_{{\rho}-{h},{\delta}-{h}}$,  and  $g_{n}$ converges to a  $(\k,\t)$-flat function $g \in  \cC^{\w,\i}_{{\rho}-{h},{\delta}-{h}}$ such that : $[H']_{{\rho}-{h},{\delta}-{h}}=0$ and $\Lambda,W,H',g$ satisfy (\ref{ered}) and (\ref{etailles}) of Proposition \ref{theo.kam}. 

This completes the proof of Proposition \ref{theo.kam}  --  except for the last analyticity 
statement. However, if  
$$\o_0\in DC(2{\k},\t)$$
then the same proof, for $s=0$, 
applied to functions in $\CC_{\rho,\delta,\delta}^{\w}$, i.e.
functions $f\in C^{\w}(\T^d_\rho \times \bD_{\delta}^d \times \bD_{\delta}^d\times \bD_{\delta})$
such that
$f(\f,r,c,\o)\in \OO^2(r,c)$
yields the analyticity of $\Lambda$, $W$ and $H'$ on $I_{\de'}$ for some $0<\de'\leq \de$. 
\subsection{Proposition \ref{theo.kam} implies Proposition \ref{counterthm}}\label{s57}

Denote by $\tilde\al(s)$ the sequence of constants in Proposition  
\ref{theo.kam}   --  
 { we can assume
without restriction that 
$$\tilde\al(s)\ge (s-t)+\tilde\al(t),\quad s\ge t,$$
-- and let
$$\a(s)=\tilde\al(s)+1+\ga, \quad \ga=10(\tau+d)+11.$$
} 
Let
$$H(\f,r)=N^q(r)+\cO^{q+1}(r),\quad q\ge 1+\al(1)$$
with $N^q(r)=\l{\o_0,r}+\cO^2(r)$. Then
\begin{multline*}
\tilde H(\f,r,c)=:H(\f,r)-N^q(c)-\l{\p_r N^q(c),r-c}\\
=a(\f,c)+\l{B(\f,c),r-c}+\cO^2(r-c)
\end{multline*}
with $a\in\cO^{q+1}(c)$ and $B\in\cO^q(c)$ { (which means in particular that $H$ is of order $q$). }
Now there exist $3\rho\ge3\de>0$ such that for all $\eta\le\de$
$$
[\tilde H]_{3\rho,3\eta,0}<C\eta^{q}$$
  ---  the constants  $\rho,\de$ and $C$ only depend on $H$.

 { Define $\sigma=\s(\eta)$ such that 
$$C\eta^{q}=\sigma\frac1{(1+\aa{\p_r^2\tilde H}_{3\rho,3\eta,0})^7}\k^{11}\eta^{\ga},$$
and note that there is a constant $C'=C'(H,\tau)$ such that if
$$\eta\le C'\k^{\frac{11}{q-\ga}},$$
then $ \sigma\le\eps(\tau)$.
}

By Proposition \ref{theo.kam} there exist $\tilde \Lambda \in \CC^{\w,\i}_{2\eta}$ and  $W\in \EE^{\w,\i}_{2\rho,2\eta}$,    
and  a $(\k,\t)$-flat function  
$g \in \cC^{\w,\infty}_{2\rho,2\eta}$   {  such that $g\in\cO^q(c)$}
$$(\tilde H+\langle \o+ \tilde \Lambda(c,\o), \cdot \rangle) \circ W_{c,\o} (\varphi,r) = 
\langle \w, r-c \rangle +  \cO^2(r-c) +  g(\f,r,c,\w)$$
(modulo an additive constant that depends on $c,\o$). Moreover, for all $s\in\N$, (\ref{etailles}) implies\footnote{\ the value of $C_s$ will change from line too line}
\begin{equation} \label{hey1}
\max(\aa{\tilde \Lambda}_{0,2\eta,s}, {[W-\id]}_{2\rho,2\eta,s}) 
< C_s\frac{\eta^{q-\ga}}{\k^{11}}(\frac{1}{\k \eta})^{\tilde \alpha(s)}.
\end{equation}

Hence if we set  $\Lambda(c,\o)=\tilde \Lambda(c,\o)-\partial_r N^q(r)$ we get that 
$$(H+\langle \o+ \Lambda(c,\o), \cdot \rangle) \circ W_{c,\o} (\varphi,r) = 
\langle \w, r-c \rangle +  \cO^2(r-c) +  g(\f,r,c,\w)$$
and, for all $s\in\N$,
\begin{equation}
\aa{\Lambda+\p_r N^q}_{0,2\eta,s}
\le C_s\frac{\eta^{q-\ga}}{\k^{11}}(\frac{1}{\k \eta})^{\tilde\alpha(s)}
\le C_s\eta^{q}(\frac{1}{\k \eta})^{\alpha(s)}.
\end{equation}

{\it The generating function.} By Proposition \ref{sympl-exact}, the diffeomorphism 
$$W(\f,r,c,\o)=(\f +\Phi(\f,c,\o),r+R_1(\f,c,\o)+R_2(\f,c,\o)(r-c))$$
has a generating function $f(\psi,r,c,\o)=f_0(\psi,c,\o)+\l{f_1(\psi,c,\o),r-c}$
$$\left\{\begin{array}{l}
s=r+ \partial_\psi f \\
\f=\psi+\p_r f=\psi+f_1. \end{array}\right.$$

If
$$\eta\le C''(H,\tau)\k^{\frac{11+\tilde\al(1)}{q-(1+\ga+\tilde\al(1))}},$$
then \eqref{hey1} implies 
$$
\aa{\Phi}_{2\rho,2\eta,1}
\le C_1\frac{\eta^{q-\ga}}{\k^{11}}(\frac{1}{\k \eta})^{\tilde\alpha(1)}\lsim \eta,
$$
and, by Proposition \ref{annexe.inverse},
$$
\aa{f_1}_{\rho,\eta,s}
\le C_s \frac{\eta^{q-\ga}}{\k^{11}}(\frac{1}{\k \eta})^{\tilde\alpha(s)}.
$$
Moreover, by Proposition \ref{annexe.compose},
$$
\aa{f_0}_{\rho,\eta,s}
\le C_s \frac{\eta^{q-\ga}}{\k^{11}}(\frac{1}{\k \eta})^{\tilde\alpha(s)},
$$
so
\begin{equation}
\aa{f}_{\rho,\eta,s}\le C_s\frac{\eta^{q-\ga}}{\k^{11}}(\frac{1}{\k \eta})^{\tilde\alpha(s)}.
\end{equation}

To conclude we observe that
$$\frac{\eta^{q-\ga}}{\k^{11}}(\frac{1}{\k \eta})^{\tilde\alpha(s)}\le
\eta^{q}(\frac{1}{\k \eta})^{\alpha(s)},$$
and that
$$\k^{\frac{\al(1)}{q-\al(1)}}\le 
\min(\k^{\frac{11}{q-\ga}},\k^{\frac{11+\tilde\al(1)}{q-(1+\ga+\tilde\al(1))}}).$$

Finally, point (iii) of  Proposition \ref{counterthm} is implied by the last statement of Proposition \ref{theo.kam}.  \carre

\section{KAM stability for Liouville tori}\label{sLiouville}


{In this section we give a sketch of the proof of  Theorem  \ref{liouville.twist} which claims 
KAM stability of a Liouville torus with a non-degeneracy condition of Kolmogorov
type. Notice that since the frequency vector is Liouville we don't have any Birkhoff normal form in general.

By assumption there exist a $\g>0$ and an increasing sequence $Q_n$  such that
$$ |\<k,\o_0\>|\ge \frac{1}{|Q_n|^{\t}}\quad \forall k\in\Z^d\sm \{0\},\ |k|\le Q_n.$$

\medskip  

\begin{lemma} \label{l91}
Let $H \in \cC^{\o}(\bT^d_{\rho}\times\bD^e_\de)$ be of the form (\ref{HH}) and let $q$ be fixed.

For any $n$ sufficiently large (depending on $H$ and $q$) , there exists an exact symplectic local diffeomorphism
$$Z(\f,r)=(\f+\cO(r),r+\cO^2(r))$$
defined in  $T^d_{\rho'}\times\bD^e_{\de'}$ where
$$\rho'\ge \rho/4\quad\textrm{and}\quad \de' \ge Q_n^{-2\ga}$$
such that 
$$H \circ Z (\varphi,r) = N^q(r)+F(\f,r) + R$$
with $N^q(r)=\l{\o_0,r}+\cO^2(r)$ and $F\in\cO^{q+1}(r)$ and
\begin{align}
|F|_{\rho',\delta'}+|N^q|_{\de'}& \le Q_n^{2\gamma q} \\
|R|_{\rho',\delta'}&\leq e^{-\sqrt{Q_n}} \label{rest} \\
\label{nondeg} \aa{\frac{\partial^2 N }{\partial r^2 } (0)-M_0}&\leq e^{-\sqrt{Q_n}}  
\end{align}
\end{lemma} 

\begin{proof} Truncate the Fourier coefficients of $H$ at order
$|k|\le Q_n'=\frac{Q_n}q$ to get $\widetilde H$ and   $H=\widetilde H+\widetilde R$.
Then
$$|\widetilde H|_{\rho/2,\delta}\le C $$
and
\begin{equation} \label{rtilde} |\widetilde R|_{\rho/2,\delta}\le C e^{-Q'_n\rho/2} \end{equation} 

Apply now Birkhoff reduction up to order $q$ to $\widetilde H$, for example as in Proposition \ref{bnfII}  with $c=r$. Indeed, the equations \eqref{cf2} (for degree $j=2$) or \eqref{cfn} (for general degree $j$) that must be solved in the construction of the degree $j$ monomials in the BNF of $\widetilde H$ involve  trigonometric polynomials on their right hand side of degree at most $(j-1) Q_n'$, $j\leq q$. Hence, we get the following estimates:

\begin{itemize}
\item 
$$|\widetilde H_j|_{\rho/2,\delta}\le C'_j$$
and
$$\displaystyle \minn_{  0<|k|\leq Q_n}  |(k,\omega_0)| \geq Q_n^{-\gamma};$$
\item it follows by a finite induction
that
$$|\G_j|_{\rho/2,\delta},\ |\Om_j|_{\rho/2,\delta},\ |G_j|_{\rho/2,\delta}\le C''_jQ_n^{(j-1)\ga}$$
and for $j\le q$
$$|f_j|_{\rho/2,\delta}\le C''_jQ_n^{j\ga};$$
\item then $Z$, implicitly defined by
$$\left\{\begin{array}{l}
\f=\psi+\frac{\p f}{\p r}(\psi,r)\\
s=r+\frac{\p f}{\p\psi}(\psi,r)
\end{array}\right. \quad f=f_2+\dots+ f_q,$$
is defined in $T^d_{\rho/4}\times\bD^e_{\de'}$ where
$$\de' \ge CQ_n^{-\ga};$$
\item $ R= \widetilde R \circ Z (\varphi,r) $ satisfies \eqref{rest} due to \eqref{rtilde} and the control on the $f_j$'s, 
\item moreover
$N^q_2(r)=\MM(\widetilde H_2(\cdot,r))$
which implies \eqref{nondeg}.
\end{itemize}
\end{proof}

\subsection{Proof of theorem D}

Fix $q=60(2d+\alpha(1)+5)$, where $\alpha(1)$ is the exponent that appears in \eqref{etailles} of Proposition \ref{theo.kam}, and apply Lemma \ref{l91} to find 
$$\bar H(\f,r)= H \circ Z (\varphi,r) = N^q(r)+F(\f,r) + R.$$
Write
\begin{multline*}
\widetilde H(\f,r,c)=:H\circ Z(\f,r)-N^q(c)-\l{\p_r N^q(c),r-c}\\
=a(\f,c)+\l{B(\f,c),r-c}+\cO^2(r-c)
\end{multline*}
with $a\in \cO^2(c)$ and $B\in\cO(c)$, i.e. $\widetilde H$ is of order $1$.

Observe that, with $\delta_n=Q_n^{-\gamma q^2}$, we have
$$[\widetilde{H}]_{\rho',\delta_n,0}\leq C(\delta_n^q Q_n^{2\gamma (q+1)} + e^{-\sqrt{Q_n}}) $$
which is $\le Q_n^{-\gamma q^3/2}=\de_n^{q/2}$ if $n$ is large enough.
 
If $\kappa_n= \delta_n^2$ and  $\tau$ is $=d$, say, then
\begin{equation} \label{Hliouville}
 [\widetilde{H}]_{\rho',\delta_n,0} \leq 
 \de_n^{q/3}{\kappa_n^{11} \delta_n}^{10(\tau+d)+11} 
 \frac1{(1+\aa{\p_r^2\widetilde{H}}_{\rho',\de_n,0})^7} \end{equation}
provided $n$ is sufficiently large. That is, \eqref{e510} is satisfied by 
$\widetilde{H}$ with $\sigma\le\de_n\le \eps(\tau)$ when $n$ is large enough.

Hence Proposition \ref{theo.kam} applies with our choice of $\kappa_n, \delta_n$ and $h=\delta_n/2$,  
yielding  $\Lambda \in \CC^{\w,\i}_{0,\de_n/2}$ and  $W\in \EE^{\w,\i}_{\rho/2,\de_n/2}$,    and  a $(\k_n,\t)$-flat function $g \in \cC^{\w,\infty}_{\rho/2,\de_2/2}$  such that

\begin{equation} \label{conj.liouv} (\bar{H}+\langle \o+ \bar{\Lambda}(c,\o), \cdot \rangle) \circ W_{c,\o} (\varphi,r) = 
\langle \w, r-c \rangle +  \cO^2(r-c) +  g(\f,r,c,\w) \end{equation}
(modulo an additive constant that depends on $c,\omega$), where we have set
$$\bar{\Lambda}(c,\o)= \Lambda(c,\omega)-\p_r N^q(c).$$
Notice that  $\Lambda(0,\om)=\om$ and that,
from \eqref{etailles} and the fact that $\sigma\le \de_n^{q/3}$, we get 
\begin{equation} \label{estimate}
\aa{\bar{\Lambda}+\p_r N^q}_{0,\delta_n/2,1} \le \de_n^2.
\end{equation}

Let $\Psi(\w,c)= \o+ \bar{\Lambda}(c,\o)$. Then $\Psi(\w_0,0)=0$ and from
\eqref{estimate} we have that 
\begin{align}
\aa{\frac{\partial \Psi}{\partial \o}(\o_0,0)-I}&\leq \de_n^2 \\
\aa{\frac{\partial \Psi}{\partial c}(\o_0,0)-M_0}&\leq 2\de_n
\end{align}
By the implicit function theorem, there exists a constant  $C(M_0)$ (that only depends on $M_0$) and a function $S: B(\o_0,C(M_0)\delta_n) \to B(0,\delta_n/2)$, such that 
$$\Psi(\o,S(\o))=0$$
Moreover $S$ is of class $C^1$ and $dS\sim M_0^{-1}$. A simple computation shows that  the set of frequencies in   $B(\o_0,C(M_0)\delta_n)$ that are $(\kappa_n, \tau)$-Diophantine has measure larger than $(1-\delta_n) \text{Leb}(B(\o_0,C(M_0)\delta_n))$ (recall that we took $\kappa_n=\delta_n^2$). 

This concludes the proof of Theorem  \ref{liouville.twist} because \eqref{conj.liouv} and the $(\kappa_n, \tau)$-flatness of $g$ imply that for any $\o \in B(\o_0,C(M_0)\delta_n) \cap \text{CD}(\kappa_n,\tau)$,  $\T^d \times \{S(\o)\}$ is an invariant KAM torus for $\bar H \circ W_{S(\o),\o}$.
}

\section{Appendix. Composition and inversion estimates.}\label{sApp}

In this Appendix we give the useful estimates for our KAM scheme.

\subsection{Convexity estimates}

\begin{Prop} \label{hadamard} 
Let $f,g \in \CC^{\w,\infty}(\T^d_\rho\times \bD^{d'}_\delta,B)$. Then

\begin{itemize}
\item[(i)]
$$
||f||_{\rho,\de,s}\le C_{s_1,s_2}||f||_{\rho,\de,s_1}^{a_1}||f||_{\rho,\de,s_2}^{a_2}$$
for all non-negative numbers $a_1,a_2,s_1,s_2$ such that
$$
a_1+a_2=1,\quad s_1a_1+s_2a_2=s.$$

\item[(ii)]
$$
||fg||_{\rho,\de,s}\ \le\ C_s(||f||_{\rho,\de,s}||g||_{\rho,\de,0}+||f||_{\rho,\de,0}||g||_{\rho,\de,s})$$
for all non-negative numbers $s$.
\end{itemize}
\end{Prop}

\begin{proof} 
A classical result  --  see the appendix of \cite{Hor} 
\end{proof}

\begin{Cor*}
Let $f,g \in \CC^{\w,\infty}(\T^d_\rho\times \bD^{d'}_\delta,B)$. Then

\begin{itemize}
\item[(i)]
$$||f||_{\rho,\de,1}^{n+1} ||f||_{\rho,\de,s-n}\le C_s
||f||_{\rho,\de,0}^{n+1}||f||_{\rho,\de,s+1}$$
for all non-negative numbers $s,n$

\item[(ii)]
$$
||f^n||_{\rho,\de,s}\ \le\ C_s^{\log(n)}||f||_{\rho,\de,s}||f||_{\rho,\de,0}^{n-1}$$
for all non-negative numbers $s,n$.
\end{itemize}
\end{Cor*}

\begin{proof} 
A computation.
\end{proof}

\subsection{Composition}

\begin{Prop} \label{annexe.compose} 
Let $f,g \in \CC^{\w,\infty}(\T^d_\rho\times \bD^{d'}_\delta,B)$ and assume that
$$\|g\|_{\rho,\delta,0}\le \frac{h}2\le \frac12\min(\rho,\de).$$
Then
$$x\mapsto f(x+g(x,\o),\o)$$
belongs to $\CC^{\w,\infty}(\T^d_{\rho-h}\times \bD^{d'}_{\delta-h},B)$ and
\begin{itemize}

\item[(i)]
$h(x,\o)=f(x+g(x,\o),\o)-f(x,\o)$ verifies
$$ \| h\|_{\rho-h,\delta-h,s}
\leq C_s\frac1{h}(\| f\|_{\rho,\de,0} \|g\|_{\rho,\de,s}+ \|  f\|_{\rho,\de,s}\|g\|_{\rho,\de,0}).$$ 

\item[(ii)]
$k(x,\o)=f(x+g(x,\o),\o)-f(x,\o)-\langle \p_x f(x,\o),g(x)\rangle$ verifies
$$ \| k\|_{\rho-h,\delta-h,s}\leq 
C_s\frac1{h^2}(\|  f\|_{\rho,\de,0} \|g\|_{\rho,\de,s}+ \|  f\|_{\rho,\de,s}\|g\|_{\rho,\de,0}) \|g\|_{\rho,\de,0}.$$ 

\end{itemize}

\end{Prop}

\begin{proof}  We will prove the statements when $x$ and $g(x,\o)$ are 
scalars. Notice
$$
f(x+g(x,\omega),\omega)=\sum_{n=0}^\infty \frac{\partial^n f}{\partial x^n}(x,\omega) \frac{g^n(x,\omega)  }{n!}.
$$

By Cauchy estimates we have for $n\geq 0$
$$ \aa{\frac{\partial^n f}{\partial x^n}}_{\rho-h,\de-h,s} \leq  \frac{1}{h^{n}} \aa{f}_{\de,s} n!$$
and, by  the  Hadamard estimates we have that
$$ \| g^n\|_{\rho,\delta,s}
\leq C_s^{\log(n)}\|g\|^{n-1}_{\rho,\de,0} \|g\|_{\rho,\de,s}.$$

Hence, for $j\ge 1$,
$$
\| \sum_{n=j}^\infty \frac{\partial^n f}{\partial x^n}(x,\omega) \frac{g^n(x,\omega)  }{n!} \|_{\rho-h,\delta-h,s}\leq$$  
$$C_s  \|  f\|_{\rho,\delta,s} 
\sum_{n\ge j}
(\frac{ \| g\|_{\rho,\delta,0}}{h})^n
+
\|  f\|_{\rho,\delta,0} \frac{ \| g\|_{\rho,\delta,s}}{h}
\sum_{n\ge j-1}
C_s^{\log(n+2)}(\frac{ \| g\|_{\rho,\delta,0}}{h})^n.
$$
\end{proof} 

\subsection{Inversion}

\begin{Prop} \label{annexe.inverse} 
Let $f \in \CC^{\w,\infty}(\T^d_\rho\times \bD^{d'}_\delta,B)$ and assume that
$$\|f\|_{\rho,\delta,1}\lsim \frac{h}2\le \frac12\min(\rho,\de).$$
Then
$$\T^d_\rho\times \bD^{d'}_\delta\ni x\mapsto \tilde f(x,\o)=x+ f(x,\o)$$
is invertible for all $\o\in B$ with an inverse 
$\T^d_{\rho-h}\times \bD^{d'}_{\delta-h}\ni y\mapsto \tilde g(y,\o)=y+g(y,\o)$ satisfying
$$\|g\|_{\rho-h,\delta-h,s} \lsim C_s \|f\|_{\rho,\delta,s} $$
for all $s\in\N$.
\end{Prop}

\begin{proof}
It is clear by the implicit function theorem that $g$ exists and that
$$\|g\|_{\rho-h,\delta-h,0}\lsim \|f\|_{\rho,\delta,0}\le h.$$
Since $g(y,\o)+ f(y+g(y,\o),\o)=0$, it follows that
$$
\p_\o g+(\p_x f)\circ \tilde g\cdot\p_\o g + (\p_\o f)\circ \tilde g=0$$
and, hence,
$$
\aa{g}_{\rho-h,\de-h,1}\lsim \aa{f}_{\rho,\delta,1}.$$

Moreover, for $n\ge1$
$$\p_\o^{n+1}g +
\p_\o^n((\p_x f)\circ\tilde g\cdot \p_\o g))+
\p_\o^{n}((\p_\o f)\circ \tilde g)=0,$$
from which we derive
\begin{multline*}
\aa{g(y,\o)}_{\rho-h,\delta-h,n+1}\lsim
\aa{(\p_x f)\circ\tilde g\cdot \p_\o g}_{\rho-h,\delta-h,n}+
\aa{(\p_\o f)\circ\tilde g}_{\rho-h,\delta-h,n},
\end{multline*}
and, by Proposition \ref{hadamard},
\begin{multline*}
\aa{g(y,\o)}_{\rho-h,\delta-h,n+1}\le C_n
\aa{(\p_x f)\circ \tilde g}_{\rho-h,\delta-h,n}\aa{f}_{\rho,\delta,1}\\
+\aa{(\p_\o f)\circ\tilde g}_{\rho-h,\delta-h,n}.
\end{multline*}

By Proposition \ref{annexe.compose}(i)
\begin{multline*}
\aa{g}_{\rho-h,\de-h,n+1}\le C_n
(\frac1h\aa{f}_{\rho,\delta,1}\aa{f}_{\rho,\delta,n}+
\aa{f}_{\rho,\delta,n+1}\\+
\aa{ f}_{\rho,\de,1}\aa{g}_{\rho-h,\de-h,n}).
\end{multline*}
By assumption $\aa{f}_{\rho,\delta,1}\lsim h$, so
$$
\aa{g}_{\rho-h,\de-h,n+1}\le C_n
(\aa{f}_{\rho,\delta,n+1}+\aa{g}_{\rho-h,\de-h,n})
$$
and the result follows by a finite induction.
\end{proof}


\begin{thebibliography}{aaaa}

\bibitem[AK]{AK}   D. ~V. Anosov and A. ~B. Katok,
\newblock New examples in smooth ergodic theory. Ergodic diffeomorphisms, 
\newblock {Transactions of the Moscow Mathematical Society} 23, 1--35, 1970.

\bibitem[B]{B} A. Bounemoura,
\newblock Generic super-exponential stability of invariant tori in Hamiltonian systems,
\newblock {Ergodic Theory Dynam. Systems} 31, 
1287--1303, 2011.

\bibitem[E]{E} L.  ~H. Eliasson,
\newblock Perturbations of stable invariant tori for Hamiltonian systems,
\newblock {Ann. Sc. Norm. Sup. Pisa} 15, 115--147, 1988.

\bibitem[FH]{FH} A. Fathi and M. Herman,
\newblock Existence de diff\' eomorphismes minimaux,
\newblock {Asterisque} 49, 
39--59, 1977.

\bibitem[FK]{FK} B. Fayad and R. Krikorian,
\newblock Herman's last geometric theorem,
\newblock {Ann. Sci. \' Ec. Norm. Sup\'er.} 42, 
193--219, 2009.

\bibitem[H]{H} M. Herman,
\newblock 
Some open problems in dynamical systems, Proceedings of the International Congress of Mathematicians, Vol. II (Berlin, 1998),
\newblock {Doc. Math. 1998 Extra Vol. II}, 797--808, 1998.

\bibitem[Ho]{Hor} L. Hormander,
\newblock 
Note  on Holder estimates.  The boundary problem of physical geodesy. Arch. Rational Mech. Anal. 62 (1976),  1--52.



\bibitem[I]{I} H. Ito,
\newblock  Convergence of Birkhoff normal forms for integrable systems,
\newblock {Comment. Math. Helv.} 64, 
412--61, 1989.

\bibitem[N]{N} Z.  Nguyen Tien,
\newblock Convergence versus integrability in Birkhoff normal form,
\newblock {Ann. of Math.} 161, 
141--156, 2005.

\bibitem[P-M]{P-M} R. Perez-Marco,
\newblock Convergence or generic divergence of the Birkhoff normal form,
\newblock {Ann. of Math.} 157, 
557--74, 2003.




\bibitem[R]{R} H. R{\"u}ssmann,
\newblock 
\"Uber die Normalform analytischer Hamiltonscher Differentialgleichungen 
in der N\"ahe einer Gleichgewichtsl\"osung,
\newblock {Math. Ann.} 169, 
55--72, 1967.

\bibitem[R2]{Russman2} H. R{\"u}ssmann,
\newblock Invariant tori in non-degenerate nearly integrable Hamiltonian systems. Regular
and Chaotic Dynamics 6:2(2001), 119-204. MR1843664 (2002g:37083)

\bibitem[S55]{S55} C.~L. Siegel,
\newblock 
\"Uber die Existenz einer Normalform analytischer Hamiltonscher 
Differentialgleichungen in der N\"ahe einer Gleichgewichtsl\"osung,
\newblock {Math. Ann} 128, 
144--70, 1954.

\bibitem[S05]{S05} L. Stolovitch,
\newblock A KAM phenomenon for singular holomorphic vector fields,
\newblock {Publ. Mat. Inst. Hautes \'Etudes Sci.} 102, 
99--165, 2005.

\bibitem[V]{V} J. Vey,
\newblock Sur certains syst\`emes dynamiques s\'eparables,
\newblock {Amer. J. Math.} 100, 
591--614, 1978.


\bibitem[XYQ]{You} 
Xu, Junxiang, You, Jiangong, Qiu, Qingjiu,
\newblock 
Invariant tori for nearly integrable Hamiltonian systems with degeneracy. Math. Z. 226 (1997), 375Ð387.

\end{thebibliography}
\end{document}